\documentclass[reqno,11pt]{amsart}
\usepackage{amssymb, amsmath}
\usepackage{graphicx}
\usepackage{cite}

\usepackage{subfig, tikz}
\usetikzlibrary{decorations.pathmorphing}
\usepackage{graphics}

\usepackage{hyperref}
\usepackage[T1]{fontenc}

\newcommand{\sech}{\,{\rm sech}\,}

\newcommand{\td}{\text{d}}

\usepackage{bm}
\usepackage{enumerate}
\theoremstyle{plain}
\newtheorem{theorem}{Theorem}[section]
\newtheorem{lemma}[theorem]{Lemma}
\newtheorem{prop}[theorem]{Proposition}
\newtheorem{cor}[theorem]{Corollary}

\newtheorem{proposition}[theorem]{Proposition}
\theoremstyle{definition}
\newtheorem{remark}[theorem]{Remark}
\newtheorem{definition}[theorem]{Definition}

\newtheorem{example}[theorem]{Example}

\usepackage{ulem}
\usepackage{soul}
\usepackage{color}

\numberwithin{equation}{section}

\def\be{\begin{equation}}
\def\ee{\end{equation}}
\def\bea{\begin{eqnarray}}
\def\eea{\end{eqnarray}}

\newcounter{mnotecount}[section]

\renewcommand{\themnotecount}{\thesection.\arabic{mnotecount}}

\newcommand{\mnote}[1]
{\protect{\stepcounter{mnotecount}}$^{\mbox{\footnotesize
$
\bullet$\themnotecount}}$ \marginpar{
\raggedright\tiny\em
$\!\!\!\!\!\!\,\bullet$\themnotecount: #1} }

\begin{document}
\title[]{On toric self-dual Einstein gravitational instantons}
\author{Bernardo Araneda}
\address[Bernardo Araneda]{
School of Mathematics and Maxwell Institute for Mathematical Sciences\\ University of Edinburgh\\
King's Buildings, Edinburgh, EH9 3JZ, UK,  \newline
Max-Planck-Institut f\"ur Gravitationsphysik (Albert-Einstein-Institut), 
Am M\"uhlenberg 1, D-14476 Potsdam, Germany} 
\email{baraneda@ed.ac.uk, bernardo.araneda@aei.mpg.de}
\author{James Lucietti}
\address[James Lucietti]{
School of Mathematics and Maxwell Institute for Mathematical Sciences\\ University of Edinburgh\\
King's Buildings, Edinburgh, EH9 3JZ, UK.}
\email{j.lucietti@ed.ac.uk}
\author{Virinchi Rallabhandi}
\address[Virinchi Rallabhandi]{
School of Mathematics and Maxwell Institute for Mathematical Sciences\\ University of Edinburgh\\
King's Buildings, Edinburgh, EH9 3JZ, UK.}
\email{v.v.rallabhandi@sms.ed.ac.uk}


\begin{abstract}
We consider the classification of  toric self-dual Einstein gravitational instantons with negative cosmological constant. As is well known, any Killing vector field on a self-dual Einstein manifold defines a local conformal K\"ahler structure. We prove that if the conformal K\"ahler structure associated to one of the torus Killing fields is global and extends to an ALE manifold with no additional fixed points,  then the corresponding self-dual Einstein instanton  is precisely given by the infinite class of multipole solutions constructed by Calderbank, Pedersen and Singer.
\end{abstract}

\maketitle

\section{Introduction}

A gravitational instanton is a complete Riemannian manifold $(M,g)$ that satisfies the Einstein equation
\begin{equation}
    {\rm Ric}_{g}=\lambda g,  \label{eq:Ein}
\end{equation}
where ${\rm Ric}_{g}$ is the Ricci tensor of $g$ and $\lambda \in \mathbb{R}$ is called the cosmological constant.
If $\lambda =0$ then there is the additional requirement than $M$ is non-compact and `asymptotically flat' in the sense of a prescribed curvature decay in some asymptotic region.  In recent years there has been impressive progress in the classification of Ricci-flat instantons that possess special geometry, culminating in an essentially complete understanding for hyper-K\"ahler instantons and near complete for  Hermitian (non-K\"ahler) instantons, see e.g. \cite{Minerbe, ChenChen, Biquard:2021gwj, Li23a, Li23b, Aksteiner:2023djq, Araneda:2025uqo, SunZhang, LiSun2}. Furthermore, generic toric Ricci-flat instantons can be classified by adapting the methods used for stationary and axisymmetric spacetimes in General Relativity (GR)~\cite{ChenTeo0, Kunduri:2021xiv, Kunduri:2026xvc}, which in particular have been used to construct remarkable new examples that violate the Riemannian no-hair theorem~\cite{ChenTeo, Teo, LiSun}.  Nevertheless, the classification of generic Ricci-flat instantons remains a notable open problem.

If $\lambda\neq 0$ the classification of Einstein metrics is an even more difficult problem and far fewer results are known. Again, it is natural to assume symmetries or special geometric structures that render this more tractable. Perhaps the most stringent --  still nontrivial -- assumption is that $(M,g)$ is self-dual and satisfies the Einstein equation \eqref{eq:Ein}. This means the Weyl tensor is (anti)self-dual and for $\lambda=0$ reduces to the aforementioned  hyper-K\"ahler instantons (locally). If $\lambda>0$ then a complete $(M,g)$ is necessarily compact (Myers' theorem) and it is known that the only  self-dual Einstein manifolds in this case are $S^4$ and $\mathbb{CP}^2$ with their canonical metrics, see e.g. \cite[Theorem 13.30]{Besse}. However, for $\lambda<0$ the classification of self-dual Einstein manifolds is open and it is the purpose of this paper to address this. 
In fact, even  toric gravitational instantons with $\lambda <0$ are much more poorly understood than their $\lambda=0$ counterparts. This is because the Einstein equations with $\lambda\neq 0$ for toric solutions are not integrable  in contrast to the $\lambda=0$ case. Indeed, this is also the case for stationary and axisymmetric spacetimes in GR and is the main reason why no analogue of the no-hair theorem is known for black holes spacetimes with $\lambda\neq 0$. 

Remarkably, Calderbank and Pedersen (CP) showed that the general local form of the metric for toric self-dual Einstein metrics with $\lambda \neq 0$ can be given in terms of a single function that is an eigenfunction of the scalar Laplacian on an auxiliary hyperbolic half-plane~\cite{Calderbank:2001uz}. Therefore,  this class of metrics is determined by a linear equation and they exploited this to write down a large class of `multipole' solutions by superposing a basic solution. Surprisingly, despite this local family of metrics being known for 20 years or so, it appears that a complete global analysis of the CP multipole solutions has not been performed. Notably, Calderbank and Singer showed that this class does include an infinite family of toric self-dual Einstein manifolds with $\lambda<0$ by exhibiting a sufficient set of conditions for the CP multipole solutions to extend to smooth metrics on complete manifolds~\cite{Calderbank:2002gy} (this builds on prior work by Joyce \cite{Joyce}).  It is well known that a Killing vector field on a self-dual Einstein manifold defines a (local) conformal K\"ahler structure that is scalar-flat~\cite{Tod1}. The Calderbank-Singer instantons were constructed by choosing the conformal K\"ahler structure with respect to one of the torus Killing fields to be a family of toric asymptotically locally Euclidean (ALE) scalar-flat K\"ahler manifolds defined on toric resolutions of $\mathbb{C}^2/\Gamma$. These generalise the structure of the Poincar\'e ball model of hyperbolic space, which recall is an open ball in $\mathbb{R}^4$ with metric conformal to the Euclidean metric (which is of course K\"ahler).  In this paper we provide a classification theorem for toric self-dual negative Einstein gravitational instantons for which one of the torus Killing fields defines a K\"ahler structure that extends to an ALE manifold. We will refer to this class of self-dual negative Einstein instantons as being {\it conformally ALE} (see Definition \ref{def:conformalALE}).  Our main result is captured by the following.

\begin{theorem}\label{thm:main}
Let $(M,g)$ be a toric conformally ALE self-dual Einstein gravitational instanton with $\lambda<0$. If the torus action has no additional fixed points in the conformal ALE manifold, then $(M,g)$ must be a Calderbank-Pedersen-Singer multi-pole instanton.
\end{theorem}

The proof relies on an alternative local form  for toric self-dual Einstein metrics due to Tod, which is specified by an axisymmetric harmonic function on Euclidean space $\mathbb{R}^3$~\cite{Tod3, Tod2}. This  form relies on the fact that a self-dual Einstein metric with a Killing field is determined by a solution to the $SU(\infty)$ Toda equation, which in the presence of a second commuting Killing field can be linearised via the Ward transformation to the Laplace equation for an axisymmetric function on Euclidean space.  The Tod form of the metric is in fact closely related to the CP metric, where the eigenfunction on the hyperbolic plane is a derivative of the axisymmetric harmonic function on Euclidean space. However, we find the Tod form of the metric more convenient for a global analysis since the problem reduces to elementary harmonic analysis on Euclidean space. In particular, we find that regularity at the axis and ALE asymptotics are sufficient to show that the solution must be a CP multipole solution. Then, a more detailed analysis at the axis ensuring the absence of conical singularities shows that if the torus action has no fixed points in the ALE extension, the only smooth solutions are precisely given by the infinite class of instantons written down by Calderbank and Singer.

It is worth noting that the Tod form of the metric is similar to the form for Ricci-flat toric Hermitian metrics, which Tod also showed can be written in terms of an axisymmetric harmonic function on Euclidean space~\cite{Tod3}.  Indeed, this form of the metric was central to the classification of toric ALF Hermitian instantons due to Biquard and Gauduchon~\cite{Biquard:2021gwj}, who exploited toric K\"ahler geometry and  harmonic analysis on Euclidean space to perform a global analysis. Subsequently, this was generalised by two of the authors to prove uniqueness of toric ALE Hermitian instantons~\cite{Araneda:2025uqo}. The latter work also simplified the analysis by avoiding the use of toric K\"ahler geometry and working directly with the Tod form of the metric. Indeed, these classifications are aided by the fact that the Tod form of the metric is automatically in the Weyl-Papapetrou coordinates that are familiar for stationary and axisymmetric spacetimes, so the boundary conditions at the axis can be implemented in terms of the {\it rod structure} which describes how the toric symmetry degenerates at the axis.  Similarly, for self-dual Einstein instantons in  Tod form, the conformally related toric scalar-flat K\"ahler metric is in Weyl-Papapetrou coordinates and hence a global analysis can be performed in an analogous way to the Ricci-flat Hermitian instantons.

It is worth noting that self-dual negative Einstein instantons are also of interest in the context of the AdS/CFT correspondence, see e.g.~\cite{Martelli:2012sz, Farquet:2014kma}. In particular, Dunajski,  Gutowski, Tod and Sabra showed that  supersymmetric solutions to four-dimensional minimal gauged Euclidean supergravity with (anti-)self-dual Maxwell field are precisely the self-dual Einstein metrics with a single Killing field~\cite{DGST1}. Therefore, our theorem gives a global classification of supersymmetric gravitational instantons in this theory, under an additional assumption of a second commuting Killing field. In fact, certain geometric and physical quantities for such toric instantons have been computed using localisation and fixed point theorems, that avoid the need for the explicit form of the metric, although assume the existence of the corresponding instantons~\cite{BenettiGenolini:2024hyd}. In contrast, our work gives the explicit form for such instantons in minimal gauged Euclidean supergravity and thus addresses the existence problem in this context. Interestingly, the Lorentzian version of this theory, namely minimal gauged supergravity, admits supersymmetric black holes such as the supersymmetric limit of the Kerr-Newman-AdS black hole~\cite{Kostelecky:1995ei, Caldarelli:1998hg}. As mentioned above, the classification of AdS black holes remains a notable open problem since the methods and integrable structures exploited in the no-hair theorems are no longer available. This is even the case for supersymmetric AdS black holes, since the general local form for a supersymmetric solution to minimal gauged supergravity reduces to a set of non-linear PDEs that do not appear integrable~\cite{Caldarelli:2003pb, Cacciatori:2004rt}. Indeed, this is similar to the general local form of supersymmetric solutions to minimal gauged Euclidean supergravity~\cite{DGST2}, except for the aforementioned special case of a (anti-)self-dual Maxwell field which reduces to the integrable self-dual Einstein case.  Therefore, our work can also be viewed as a Euclidean version of a classification of supersymmetric AdS black holes.

The organisation of this paper is as follows. In Section \ref{sec:local} we review the local form of toric self-dual Einstein metrics, including both the Tod form and the CP form, and the example of multipole solutions. In Section \ref{sec:asymptotics} we introduce the class of conformally ALE self-dual Einstein manifolds and deduce the corresponding asymptotic boundary conditions in the Tod form of the metric. Finally, in Section \ref{sec:globalAnalysis} we perform a global analysis of the Tod form of the metric, including a detailed regularity analysis of the fixed point sets of torus Killing fields, and prove Theorem \ref{thm:main}.

\subsection*{Acknowledgements}
BA is supported via the ERC Consolidator/UKRI Frontier grant TwistorQFT EP/Z000157/1. VR is supported by a School of Mathematics Studentship from the University of Edinburgh.

\section{Local form of toric self-dual Einstein metrics}\label{sec:local}

\subsection{Preliminaries}
Let $(M,g)$ be a four-dimensional, orientable, Riemannian manifold. The Weyl tensor of $g$, denoted $\mathcal{W}$, splits under the Hodge star operator $\star_{g}$ as $\mathcal{W}=\mathcal{W}^{+}+\mathcal{W}^{-}$, where $\mathcal{W}^{\pm}=\frac{1}{2}(\mathcal{W}\pm\star_{g}\mathcal{W})$ is called the self-dual/anti-self-dual Weyl tensor. In abstract indices, $\star_{g}\mathcal{W}_{abcd}=\frac{1}{2}\varepsilon_{ab}{}^{mn}\mathcal{W}_{mncd}$, where $\varepsilon_{abcd}$ is the volume 4-form and indices are raised and lowered with $g$ and its inverse. We say that $(M,g)$ is (anti)-self-dual if $\mathcal{W}=\pm\star_{g}\mathcal{W}$ and Einstein if \eqref{eq:Ein} is satisfied.
We say that $(M,g)$ is a self-dual Einstein manifold if it is both (anti)-self-dual and Einstein. We choose the convention that $(M,g)$ is self-dual if $\mathcal{W}^+=0$ (so in fact $\mathcal{W}$ is anti-self-dual). We shall be interested in $\lambda<0$ and hence we will set
\begin{equation}
    \lambda=-\frac{3}{\ell^2}  \; ,  \label{eq:lambda}
\end{equation}
where $\ell$ is a real non-zero constant, although the local results in this section also hold for $\lambda>0$ with the appropriate replacements.

From the work of Przanowski \cite{Przanowski} and Tod \cite{Tod1, Tod2}, we know that, if $(M,g)$ is self-dual Einstein and has a Killing vector  $\xi$, then it can be described in terms of the $SU(\infty)$ Toda formulation, as follows.  
Define a 2-form $Z$ on $M$ and a scalar field $z$ on the region $|Z|_g \neq 0$ by 
\begin{align}\label{zToda}
Z = \frac{1}{2}(\td\xi^{\flat}+\star_{g}\td\xi^{\flat}), 
\qquad 
z = \frac{2}{|Z|_{g}} ,
\end{align}
where $\xi^{\flat}=g(\xi,\cdot)$.
Then, it can be shown that $\mathcal{W}^+=0$ implies that $Z$ is a conformal Killing Yano tensor (CKY) and that $J^{a}{}_{b}=-z Z^{a}{}_{b}$ is an integrable almost-complex structure compatible with $g$. Thus $(M,g, J)$ is locally Hermitian with fundamental form $\omega_{ab}= J^{c}_{~a} g_{cb} =  z Z_{ab}$, which also satisfies $\iota_\xi \omega = \td z^{-1}$ and hence both $z$ and $J$ are also invariant under $\xi$. Furthermore, $\hat{g}= z^2 g$ is a scalar-flat K\"ahler metric with K\"ahler form $\hat\omega= z^2 \omega$, i.e., the complex structure $J$ is parallel with respect to the Levi-Civita connection of $\hat{g}$, such that $\iota_\xi \hat\omega= -\td z$ so the Killing field $\xi$ is also Hamiltonian.

One can introduce a coordinate system $(\psi,x,y,z)$ adapted to this structure such that away from any fixed points of the Killing field $\xi= \partial_\psi$ the metric on the region $z>0$ takes the local form
\begin{align}\label{Todametric}
g = \frac{1}{z^2} \left[ W^{-1}(\td\psi+A)^2 + W\left(\td{z}^2+e^{u}(\td{x}^2+\td{y}^2)\right) \right],
\end{align}
where $W$, $u$ are functions and $A$ a 1-form on the space of orbits of $\xi$ with coordinates $(x,y,z)$, satisfying the following system:
\begin{subequations}\label{Todasystem}
\begin{align}
0 ={}& u_{xx}+u_{yy}+(e^{u})_{zz}, \label{Todaeq} \\
\td A ={}& W_{x}\td{y}\wedge\td{z} + W_{y}\td{z}\wedge\td{x} + (We^{u})_{z}\td{x}\wedge\td{y}, \label{monopoleeq} \\
W ={}& \ell^2 \left(1-\frac{zu_{z}}{2} \right). \label{W}
\end{align}
\end{subequations}
Equation \eqref{Todaeq} is called the Toda equation. The K\"ahler metric in these coordinates is 
\begin{align}\label{SFKmetric}
\hat{g} := W^{-1}(\td\psi+A)^2 + W\left(\td{z}^2+e^{u}(\td{x}^2+\td{y}^2)\right).
\end{align}
 Observe that $\hat{g}$ belongs to the LeBrun's class of scalar-flat K\"ahler metrics with symmetry \cite{LeBrun91}, as can be seen from \eqref{SFKmetric}, \eqref{Todaeq} and \eqref{monopoleeq}.

\begin{remark}
    It is worth noting that the above local coordinate structure has the scaling freedom
\begin{equation}
(\psi, z, x, y, e^u) \mapsto ( \alpha \psi, \alpha z, x, y, \alpha^2 e^u) \; , \label{eq:scaling}
\end{equation}
where $\alpha>0$ is a constant, which leaves $g, \omega$ invariant.
\end{remark}

\subsection{Toric symmetry}

A Riemannian manifold $(M,g)$ is toric if there is an (effective) action of the torus $T^2=S^1\times S^1$ on $M$ that preserves $g$. Henceforth we will assume that $(M, g)$ is toric and self-dual Einstein, such that $\xi$ lies in the span of the torus Killing fields. Therefore, in the above we now assume that there is a second Killing field $\eta$ that commutes with $\xi$.

A local description of all self-dual Einstein metrics with two commuting isometries was given by Calderbank and Pedersen \cite{Calderbank:2001uz}, in terms of an eigenfunction of the Laplacian on the hyperbolic plane (see Remark \ref{remark:CP}).
An alternative description  was given by Tod \cite{Tod3}, which reduces the problem to the axisymmetric Laplace equation on $\mathbb{R}^3$, and is more convenient for our purposes. 
In terms of the system \eqref{Todametric}-\eqref{Todasystem}, Tod showed that the second Killing vector can be always chosen to be $\eta=\partial_{y}$, so the Toda equation \eqref{Todaeq} reduces to 
\begin{align}\label{Todareduced}
u_{xx}+(e^{u})_{zz}=0,
\end{align}
and the 1-form $A$ is $A=F\td{y}$ for some function $F=F(x,z)$. The equation \eqref{Todareduced} implies that locally there is a function $\zeta(x,z)$ such that $u_{x}=-2\zeta_{z}$ and $(e^{u})_{z}=2\zeta_{x}$. Define the function $\rho 
> 0$ by $\rho^2:=e^{u}$. Now consider the coordinate change $(x,z)\to(\rho,\zeta)$. By inspecting the inverse of the Jacobian matrix one deduces that $z_{\rho}=\rho x_{\zeta}$ and $z_{\zeta}=-\rho x_{\rho}$: the second equation implies that there is a function $V=V(\rho,\zeta)$ defined by 
\begin{align}\label{definitionV}
z = \tfrac{1}{2}\rho V_{\rho}, \qquad x = -\tfrac{1}{2}V_{\zeta}, 
\end{align}
whereas from the first equation we see that $V$ solves 
\begin{align}\label{Laplace}
\rho^{-1}(\rho V_{\rho})_{\rho} + V_{\zeta\zeta} = 0  \; .
\end{align}
Therefore $V$ is an axisymmetric harmonic function on Euclidean space with cylidrical cooordinates $(\rho, \zeta)$.
Changing coordinates in \eqref{Todametric}-\eqref{Todasystem} to $(\psi, y,\rho,\zeta)$ we then arrive at the following crucial result.

\begin{proposition}[Tod \cite{Tod3}]\label{prop:Tod}
Let $(M,g)$ be a self-dual Einstein manifold with two commuting Killing vector fields. Then the metric can be written in the local form 
\begin{align}\label{Todmetric}
g = \frac{1}{z^2}\left[W^{-1}(\td\psi+F\td{y})^2+W\rho^2\td{y}^2+e^{2\nu}(\td\rho^2+\td\zeta^2) \right],
\end{align}
where the Killing fields are $\xi=\partial_{\psi}$ and $\eta= \partial_{y}$, the functions $z,W,F,\nu$ are given by
\begin{equation}\label{functionsTodmetric}
\begin{aligned}
 z ={}&  \tfrac{1}{2}\rho V_{\rho}, \\
 W ={}&\ell^2\left(1+\frac{V_{\rho}V_{\zeta\zeta}}{\rho(V_{\zeta\zeta}^2+V_{\rho\zeta}^2)} \right), \\
 F ={}& \ell^2 \left(\zeta - \frac{V_{\rho}V_{\rho\zeta}}{V_{\zeta\zeta}^2+V_{\rho\zeta}^2} \right), \\
 e^{2\nu} ={}& \frac{W\rho^2}{4}(V_{\zeta\zeta}^2+V_{\rho\zeta}^2),
\end{aligned}
\end{equation}
and $V$ is a solution to the axisymmetric Laplace equation \eqref{Laplace} on $\mathbb{R}^{3}$.
\end{proposition}

\begin{remark}\label{rem:scaling2}
    The above local metric has the following scaling freedom
    \begin{equation}
        (V, \rho, \zeta, \psi, y) \to ( \alpha V, \alpha^{-1} \rho, \alpha^{-1} \zeta, \alpha \psi, \alpha^2 y)
    \end{equation}
    where $\alpha>0$ is a constant, which leaves $g$ invariant.
\end{remark}

The coordinate system $(\psi, y,\rho,\zeta)$ is adapted to the toric symmetry and gives rise to a geometrically defined coordinate system on the orbit space $\mathcal{O}=M/T^2$, as follows. If $G$ denotes the Gram matrix of the Killing fields $(\partial_\psi, \partial_y)$, then it is easily verified that 
\begin{align}\label{determinant}
\det G=\frac{\rho^2}{z^4}.
\end{align}
The metric induced on the orbit space is $q= z^{-2} e^{2\nu}(\td \rho^2+ \td \zeta^2)$ and hence $\rho$ is harmonic on $(\mathcal{O}, q)$, while $\zeta$ is its harmonic conjugate defined by $\td \zeta = - \star_q \td \rho$ (choosing an orientation).  In fact, with respect to the K\"ahler metric $\hat{g}$, the determinant of the Gram matrix and orbit space metric are
\begin{align}
    \det \hat G = \rho^2, \qquad \hat{q}= e^{2\nu}(\td \rho^2+ \td \zeta^2)  \label{eq:WPcoords}
\end{align}
and therefore we see that this coordinate system precisely corresponds to the Weyl-Papapetrou coordinates familiar in GR. Indeed, any scalar-flat K\"ahler metric corresponds to an Einstein-Maxwell solution~\cite{Flaherty} and hence toric solutions in this class admit Weyl-Papapetrou coordinates defined by \eqref{eq:WPcoords}.

\begin{remark}\label{remark:CP}
Calderbank-Pedersen showed by different methods that any toric self-dual Einstein structure  can be written in the local form (setting $\ell=1$)~\cite{Calderbank:2001uz} 
\begin{align}\label{CPmetric}
g = \frac{\rho([z_{\zeta}\td\psi+(\rho z_{\rho} + \zeta z_{\zeta}-z)\td y]^2+[z_{\rho}\td\psi+(\zeta z_{\rho} - \rho z_{\zeta})\td y]^2)}{z^2[\rho(z_{\rho}^2+z_{\zeta}^2)-zz_{\rho}]} 
+ \frac{\rho [\rho(z_{\rho}^2+z_{\zeta}^2)-zz_{\rho}]}{z^2} h
\end{align}
where $\partial_\psi, \partial_y$ are the torus Killing fields,  $\rho>0$,  $h= (\td \rho^2+ \td \zeta^2)/\rho^2$ is the metric on the hyperbolic half-plane, 
\begin{align}\label{zCP}
z(\rho,\zeta) = \sqrt{\rho}\mathcal{F}(\rho,\zeta) 
\end{align}
and $\mathcal{F}=\mathcal{F}(\rho,\zeta)$ is a positive function satisfying 
\begin{align}\label{eigenfunction}
\mathcal{F}_{\rho\rho}+\mathcal{F}_{\zeta\zeta} = \frac{3}{4\rho^2}\mathcal{F} \; .
\end{align}
Thus $\mathcal{F}$ is an eigenfunction of the Laplacian on the hyperbolic half-plane with eigenvalue $3/4$. Using \eqref{definitionV}-\eqref{Laplace}, one can check that the metrics \eqref{CPmetric} and \eqref{Todmetric}-\eqref{functionsTodmetric} are indeed the same. In particular, note that the relation between the axisymmetric harmonic function $V$ and the eigenfunction on hyperbolic space $\mathcal{F}$ is given by 
\begin{equation}\label{eq:FV}
    \mathcal{F}= \tfrac{1}{2} \sqrt{\rho} V_\rho  \; .
\end{equation}
Therefore, fixing $\mathcal{F}$ fixes $V$ up to an affine function of $\zeta$ (using harmonicity), which is precisely the redundancy in the Tod form of the metric.
\end{remark}

\begin{example}[Hyperbolic space]\label{ex:hyperbolic}
In the  Poincar\'e ball model of hyperbolic space $M$ is a unit open ball in $\mathbb{R}^4$ with metric
\begin{equation}
g = \frac{4}{(1-r^2)^2} \left( \td r^2 +  \frac{r^2}{4} \left( (\td \psi+ \cos\theta \td \phi)^2+ \td \theta^2+ \sin^2\theta \td \phi^2 \right) \right) \; , \label{H4}
\end{equation}
where $r<1$, the unit metric on $S^3$ is written in Euler angles $(\theta, \psi, \phi)$, and we have fixed $\ell=1$.  Choosing the Killing vector $\xi= \partial_\psi$ we can compute the self-dual two-form $Z$ defined in \eqref{zToda} (fixing orientation $\epsilon_{\psi r \theta \phi}>0$) and  deduce that the coordinate change to the Toda form of the metric \eqref{Todametric} is\begin{equation}
z= 1-r^2, \qquad \tanh x = \cos \theta , \qquad y=\phi \; ,
\end{equation}
where 
\begin{equation}
W= \frac{1}{1-z}, \qquad e^u = (1-z)^2\sech^2 x , \qquad A=  \tanh x \, \td y \; .
\end{equation}
Note the coordinate ranges are $0<z\leq 1$ and $x\in \mathbb{R}$. In particular, the conformal boundary of hyperbolic space is a unit $S^3$ at $r=1$ and is given by $z=0$. Furthermore, note that the conformal K\"ahler metric $\hat{g}= z^2 g$ is that of Euclidean space.

We can also deduce that Weyl-Papapetrou coordinates in the Tod form of the metric \eqref{Todmetric} are 
\begin{equation}
\rho = (1-z)\sech x, \qquad \zeta = (1-z) \tanh x
\end{equation}
and hence  $\rho^2+ \zeta^2 = (1-z)^2<1$. 
Thus, hyperbolic space corresponds to the interior of the unit semi-circle centred at the origin of the $(\rho, \zeta)$-plane, with the conformal boundary $z=0$ corresponding to the semi-circle $\rho^2+\zeta^2=1$. Inverting the coordinate change gives
\begin{equation}\label{eq:zH}
z = 1- \sqrt{\rho^2+\zeta^2}, \qquad x = \sinh^{-1} \left( \frac{\zeta}{\rho}\right)  \; .
\end{equation}
\end{example}

\begin{example}[Self-dual negative Einstein Taub-bolt]\label{ex:TB}
The Taub-NUT solution with $\lambda=-3/\ell^2$ and anti-self-dual Weyl tensor is
\begin{equation}
g = \frac{\Delta}{r^2-n^2} ( \td \tau+ 2n \cos \theta \td \phi)^2+ \frac{r^2-n^2}{\Delta} \td r^2 + (r^2-n^2) (\td \theta^2+\sin^2\theta \td \phi^2)
\end{equation}
where $n$ is a constant,
\begin{equation}
\Delta = (r + n)^2 \left( 1- \frac{4n^2}{\ell^2} + \frac{1}{\ell^2} (r -  n)^2 \right)
\end{equation}
and the orientation is $\epsilon_{\tau r \theta \phi}>0$. 
The roots of $\Delta$ are: $-n$ (double) and $r_\pm = n \pm \sqrt{4n^2 -\ell^2}$. If $n>\ell/2$ then $r_+>n$ is the largest real root and $g$ is smooth for $r>r_+$ and we can add a bolt at $r=r_+$ smoothly provided $\tau\sim \tau +8\pi n/p$ and 
\begin{equation}\label{nTB}
    n= \frac{\ell p}{4\sqrt{p-1}}  \; ,
\end{equation}
where $p\geq 3$ is an integer. This gives a complete metric on $M=\mathcal{O}(-p)$ (complex line bundle over $S^2$ with Chern number $p$). The constant $r$ surfaces are lens spaces $L(p,1)$ and there is a conformal boundary as $r\to \infty$. On the other hand, if $n<0$, then $g$ extends to a smooth Einstein metric on $\mathbb{R}^4$ with a NUT at $r=-n$ which generalises the Euclidean Taub-NUT solution.

The Toda data with respect to $\xi= \partial_\tau$ is
\begin{align}
&z= \frac{1}{r+ n}, \qquad  \tanh x = \cos\theta, \qquad y=\phi, \qquad A= 2n \tanh x\td y \\
& W = \frac{(1-2n z)\ell^2}{1- 4n z +\ell^2 z^2} , \qquad e^u = \frac{1}{\ell^2} (1- 4n z+ \ell^2 z^2) \text{sech}^2 x  \; ,
\end{align}
where $z=0$ is the conformal boundary.  If $n<0$, the above $z$ is singular at the NUT and hence $g$ is not globally conformally K\"ahler with respect to this structure. However, if $n>\ell/2$, then $0<z\leq 1/(r_++n)$ so this solution is globally conformally K\"ahler with respect to $\xi$, so we will focus on this solution.

The Weyl-Papapetrou coordinates are
\[
\rho = \frac{1}{\ell} \sqrt{1- 4n z+ \ell^2 z^2} \; \text{sech} x, \qquad \zeta = \left( z- \frac{2n}{\ell^2}\right) \tanh x,
\]
the conformal boundary $z=0$ corresponds to the semi-ellipse
\[
\rho^2 + \frac{\ell^2}{4 n^2} \zeta^2 = \frac{1}{\ell^2}
\]
and the domain corresponds to the interior of the semi-ellipse in the $(\rho, \zeta)$-plane.  We can invert this to obtain
\begin{equation}\label{eq:zTB}
z  =  \frac{2n}{\ell^2} -   \frac{1}{2} \sqrt{ \rho^2+ (\zeta- \sqrt{c})^2} -    \frac{1}{2} \sqrt{ \rho^2+ (\zeta+\sqrt{c})^2}
\end{equation}
where $c:= \ell^{-4}( 4n^2-\ell^2)>0$.
\end{example}

\begin{example}[Multipole solutions]\label{ex:multi} Calderbank and Pedersen introduce the $(n+1)$-pole solutions~\cite{Calderbank:2001uz}  
\begin{equation}
    \mathcal{F}= \frac{A}{\sqrt{\rho}}+ \sum_{i=1}^{n} a_i \mathcal{F}_0( \rho, \zeta-z_i)  \label{eq:multiF}
\end{equation}
where 
\begin{equation}
\mathcal{F}_0(\rho, \zeta):= \frac{\sqrt{\rho^2 +\zeta^2}}{\sqrt{\rho}}
\end{equation}
is a basic solution to \eqref{eigenfunction} and $A, a_i, z_i$ are constants. The corresponding axisymmetric-harmonic function is found from \eqref{eq:FV} (up to an affine function of $\zeta$) to be
\begin{equation}
    V=A \log \rho^2+  \sum_{i=1}^{n} a_i V_0( \rho, \zeta-z_i)
\end{equation}
where 
\begin{equation}
V_0( \rho, \zeta) := 2 \sqrt{\rho^2+\zeta^2} - \zeta \log \left( \frac{\sqrt{\rho^2+\zeta^2} + \zeta }{\sqrt{\rho^2+\zeta^2} - \zeta } \right) \label{eq:V0}
\end{equation}
is the axisymmetric-harmonic function defined by the basic solution $\mathcal{F}_0$. For later purposes it is useful to note that
\begin{equation}\label{eq:V0axis}
V_0(\rho, \zeta)= |\zeta|\log \rho^2+O(1)
\end{equation}
as $\rho\to 0$.

By inspecting $z$ for hyperbolic space \eqref{eq:zH} and self-dual negative Einstein Taub-bolt \eqref{eq:zTB} and using \eqref{zCP}, we can immediately deduce that they are $2$-pole and $3$-pole solutions respectively. The $(n+1)$-pole solutions are a large local family of self-dual Einstein metrics \eqref{Todmetric}, \eqref{functionsTodmetric}, whose global analysis has yet to be fully performed. In this paper, we will show that the classification of a class of self-dual Einstein manifolds reduces to a global analysis of the multipole solutions.

\end{example}

\subsection{Change of Killing field}\label{subsec:changeKVF}

The Tod form for a toric self-dual Einstein structure given in Proposition \ref{prop:Tod} is not symmetric with respect to the toric Killing fields $\partial_\psi, \partial_y$.
However, the toric symmetry does not single out a preferred Killing vector: any linear combination 
\begin{equation}
    \tilde{\xi}=b_{1}\partial_{\psi}+b_2\partial_{y} \; ,
\end{equation} 
with $b_1, b_2\in \mathbb{R}$, defines an integrable complex structure via \eqref{zToda} and there will be corresponding descriptions in terms of new adapted coordinates $(\tilde\psi,\tilde{y},\tilde\rho,\tilde\zeta)$ where $\tilde\xi=\partial_{\tilde\psi}$ and a new scalar field $\tilde{z}$. 
In particular, note that this means that there is a two-dimensional family, parameterised by $b_1,b_2$, of scalar-flat K\"ahler metrics conformal to the Einstein metric $g$.

To relate the different descriptions consider a new basis of toric Killing fields $(\tilde{\xi}, \tilde \eta)$ and let $\tilde G$ denote its Gram matrix. Since the two bases are related by $(\partial_\psi, \partial_y) =(\tilde{\xi}, \tilde \eta)L$  for some $L\in GL(2, \mathbb{R})$, it follows that $\det G= a^2\det\tilde{G}$ where $a=|\det L| > 0$, so from \eqref{determinant} we deduce that $a \tilde\rho/\tilde{z}^2 = \rho/z^2$. Note that from \eqref{zCP} this implies that the eigenfunction $\tilde{\mathcal{F}}= \sqrt{a} \mathcal{F}$ and so is invariant under changes of Killing field up to an overall multiplicative constant. For simplicity we will set $\ell=1$ in the rest of this section.

Using \eqref{CPmetric}, a calculation shows that the scalar field $\tilde{z}$ associated to $\tilde\xi=b_{1}\partial_{\psi}+b_2\partial_{y}$ via \eqref{zToda} is 
\begin{align}\label{eq:ztilde}
\tilde{z} = \frac{z}{\sqrt{(b_1+b_2\zeta)^2+b_2^2\rho^2}}.
\end{align}
We can then find $\tilde\rho$ and its harmonic conjugate $\tilde\zeta$: a short computation gives
\begin{align}\label{tilderz}
\tilde\rho = \frac{\rho}{a[(b_1+b_2\zeta)^2+b_2^2\rho^2]}, 
\qquad 
\tilde\zeta = c - \frac{(b_1+b_2\zeta)}{a b_2[(b_1+b_2\zeta)^2+b_2^2\rho^2]},
\end{align}
where $c$ is a constant. Alternatively, in terms of the complex coordinate $w=\zeta+i\rho$ and $\tilde{w}=\tilde\zeta+i\tilde\rho$, \eqref{tilderz} is equivalent to 
\begin{align}
\tilde{w} = \frac{a_1+a_2 w}{b_1+b_2 w},
\label{eq:wTilde}
\end{align}
with $a_1 :=cb_1-a^{-1}b_2^{-1}$, $a_2:=cb_2$ and note $a_2 b_1- a_1 b_2 =a^{-1}>0$. Thus, a change of Killing field is equivalent to a $PGL(2,\mathbb{R})$ transformation in the upper-half $w$-plane (see also~\cite{Farquet:2014kma}).

\begin{example}[Multipole solutions under change of Killing field]\label{ex:multipolechange}
Consider the transformation of the multipole solutions in Example \ref{ex:multi}  under a change of preferred Killing field from $\xi$ to $\tilde{\xi}$ with $b_2\neq 0$ as above.  Then, from equation \eqref{tilderz} we find the basic solution transforms as
\begin{align}
    \mathcal{F}_0(\rho, \zeta - z_i) &= \sqrt{a}|b_1 + b_2z_i|\mathcal{F}_0(\tilde{\rho}, \tilde{\zeta} - \tilde{z}_i) \; ,
\end{align}
where 
\begin{align}
    \tilde{z}_i &= c - \frac{1}{ab_2(b_1 + b_2z_i)}
\end{align}
correspond to the $PGL(2, \mathbb{R})$ transformation of $z_i$.  Furthermore, taking the inverse transformation of equation \eqref{eq:wTilde} shows
\begin{align}
    \rho &= \frac{\tilde{\rho}}{ab_2^2((\tilde{\zeta} - c)^2 + \tilde{\rho}^2)}.
\end{align}
Therefore, recalling that $\tilde{\mathcal{F}}= \sqrt{a} \mathcal{F}$, we find that \eqref{eq:multiF} transforms to
\begin{align}
    \tilde{\mathcal{F}} = aA|b_2|\mathcal{F}_0(\tilde{\rho}, \tilde{\zeta} - c) + a\sum_{i = 1}^na_i|b_1 + b_2z_i|\mathcal{F}_0(\tilde{\rho}, \tilde{\zeta} - \tilde{z}_i).
    \label{eq:newMultipole}
\end{align}
We deduce that a multipole solution is transformed into another multipole solution by change of preferred Killing field.

As a special case, consider hyperbolic space. As remarked above, hyperbolic space is a 2-pole solution, which from \eqref{eq:zH} and \eqref{zCP} is given by 
\begin{equation}
    \mathcal{F} = \frac{1}{\sqrt{\rho}} - \mathcal{F}_0(\rho, \zeta) \; ,
\end{equation} 
i.e. $n = 1$, $A = 1$, $z_1 = 0$ and $a_1 = -1$. Suppose we perform a change of preferred Killing field with
\begin{align}
    L^{-1} &= \begin{bmatrix}
        1 & 1 \\
        -1 & 1
    \end{bmatrix} \,\,\,\mathrm{and}\,\,\, c = -1.
\end{align}
Then, $a = 1/2$, $b_1 = 1$, $b_2 = -1$, $\tilde{z}_1 = 1$ and therefore
\begin{align}
    \tilde{\mathcal{F}} &= \frac{1}{2}\left(\mathcal{F}_0(\tilde{\rho}, \tilde{\zeta} + 1) - \mathcal{F}_0(\tilde{\rho}, \tilde{\zeta} - 1)\right),
\end{align}
which is the form of hyperbolic space given in \cite{Calderbank:2002gy, Farquet:2014kma}. The specific choices for $L$ and $c$ made in this example are canonical in a sense explained in Section \ref{sec:globalAnalysis} - see equations \eqref{eq:aInverse} and \eqref{eq:c}.
\end{example}

\section{Asymptotic structures}\label{sec:asymptotics}

\subsection{Conformally compact manifolds}  

\begin{definition}[See e.g.~\cite{Anderson}]
A Riemannian manifold $(M,g)$ is conformally compact if:
\begin{itemize}
\item There is a compact manifold with boundary $\bar{M} = M \cup \partial M$ where $\partial M$ is compact.
\item There exists a smooth function $w$ on $\bar{M}$ such that (i) $w>0$ on $M$ (ii) $w=0$ and $\td w \neq 0$ on $\partial M$; (iii) $\bar{g}=w^2 g$ extends to a smooth metric on $\bar{M}$.  
\end{itemize}
The function $w$ is called a boundary defining function, and is defined only up to positive non-constant rescalings.
\end{definition} 

Now, suppose that $(M,g)$ is toric self-dual Einstein with negative cosmological constant \eqref{eq:lambda}. 
Recall that the geometry can be described in  terms of the Toda system \eqref{Todametric}-\eqref{Todasystem} where  $u_y=0$ (recall the toric Killing fields are $\xi=\partial_\psi$ and $\eta=\partial_y$). For analytic solutions to the Toda equation \eqref{Todaeq}, we can Taylor expand near $z=0$ as
\begin{equation}
\begin{aligned}
u(x,z) ={}& u^{(0)}(x) + z\, u^{(1)}(x) +O(z^2) \\
W(x,z) ={}& \ell^2\left(1-\tfrac{1}{2}z \, u^{(1)}(x) +O(z^2) \right) \\
F(x,z)={}& F^{(0)}(x)+ \tfrac{1}{2}\ell^2 u^{(1)}(x)z+ O(z^2) , \qquad \partial_x F^{(0)}= \tfrac{1}{2} \ell^2 e^{u{(0)}} u^{(1)} \; ,
\end{aligned}
\end{equation}
where to obtain the second and third equation we used \eqref{W} and \eqref{monopoleeq}. Thus, replacing into \eqref{Todametric}, we find 
\begin{align}
g = \frac{1}{z^2} \left[\ell^2 \td{z}^2+\frac{1}{\ell^2}(\td\psi+F^{(0)}\td y)^2+\ell^2 e^{u^{(0)}}(\td{x}^2+\td{y}^2) \right] + O(z^{-1}),
\end{align}
and hence we deduce that the K\"ahler metric $\hat g = z^2 g$
extends smoothly to $z=0$ and to at least small $z<0$. We deduce that $(M,g)$ is conformally compact, with a boundary defining function $z$ and a conformal boundary $\{z=0\}$ (this was also shown in~\cite{Farquet:2014kma}). The induced metric on the boundary $\partial M$ is 
\begin{equation}
h= \frac{1}{\ell^2}(\td\psi+F^{(0)}\td y)^2+\ell^2 e^{u^{(0)}}(\td{x}^2+\td{y}^2)
\end{equation}
and is fully determined by the Toda data $(u^{(0)}(x), u^{(1)}(x))$.  The boundary inherits the torus symmetry and therefore by compactness  must be $S^3$, a lens space $L(p,q)$, $S^1\times S^2$ or $T^3$.  It is straightforward to see that any analytic solution to the Toda equation \eqref{Todareduced} is uniquely determined by $(u^{(0)}(x), u^{(1)}(x))$. We deduce that any toric self-dual Einstein metric with $\lambda<0$ is uniquely determined by its boundary metric.  However, from this perspective it is unclear which boundary data $(u^{(0)}(x), u^{(1)}(x))$ leads to complete metrics.  Therefore, we will take an alternative approach.

\subsection{Conformally ALE manifolds}

As shown above, we can extend $M$ through the boundary defined by $z=0$ to a region $z<0$. It is then natural to consider extensions $(\hat{M}, \hat g)$, so $M\subset \hat M$, that are complete K\"ahler manifolds.  We will find that it is natural to consider extensions that are ALE. It is thus convenient to first recall the definition of ALE.

\begin{definition}[See e.g. Def. 3.1 in \cite{Araneda:2025uqo}]
A Riemannian manifold $(\hat{M},\hat{g})$ is asymptotically locally Euclidean (ALE) of order $\tau>0$ if there is an end diffeomorphic to $\mathbb{R}\times S$, with $S=S^3/\Gamma$ and $\Gamma$ a finite subgroup of $O(4)$ acting freely, and the metric on the end has the asymptotic behaviour
\begin{align}\label{gALE}
\hat{g} = \td{r}^2+r^2\gamma+h,
\end{align}
where $\gamma$ is a Riemannian metric on $S$ with constant curvature $+1$, and $|\nabla^{k}h|=O(r^{-\tau-k})$ as $r\to\infty$ for all $k\geq0$, where the covariant derivative and norm are with respect to $\td{r}^2+r^2\gamma$.

 We say that $(\hat{M},\hat{g})$ is toric ALE  if it is toric and ALE and and the $T^2$-action is also an isometric action on the end $(S,\gamma)$. This implies that $S$ is a lens space $L(p,q)$.
\end{definition}

It is helpful to  consider some examples.  The simplest is hyperbolic space in Example \ref{ex:hyperbolic}, from which we can immediately see that an extension is $\hat M=\mathbb{R}^4$ with $\hat g$ given by the Euclidean metric. Therefore, the extension in this case is AE.  A more nontrivial example follows.

\begin{example}
    Consider the self-dual negative Einstein Taub-bolt solution in Example \ref{ex:TB}. Using $z=1/(r+n)$ as a coordinate, the K\"ahler metric is 
\begin{align}
\hat{g} &=\frac{\td z^2}{F(z)}+  F(z) (\td \tau+ 2n \cos\theta \td \phi)^2+  (1-2n z)(\td \theta^2+\sin^2\theta \td \phi^2), \\ &F(z) := \frac{ 1- 4n z+ \ell^2 z^2}{\ell^2 (1-2n z)} \; ,
\end{align}
where recall $0<z \leq  1/(r_++n)$ and $n>\ell/2$.  The hypersurface $z=0$ corresponds to the conformal boundary $\partial M$ of the Einstein manifold which is $L(p,1)$ with metric
\[
h= \frac{1}{\ell^2} (\td \tau+ 2n \cos\theta \td \phi)^2+  \td \theta^2+\sin^2\theta \td \phi^2 \; .
\]
In fact, the K\"ahler metric extends to a smooth metric for all $z\leq 0$ since $F(z)>0$ in that region. 
 Interestingly, this metric is ALE as $z\to - \infty$. To see this set $1+2n z= \tfrac{1}{4}R^2$ and one finds that 
\[
\hat g\sim \td R^2 +\tfrac{1}{4} R^2 [  ( \td \tau/(2n) + \cos\theta \td \phi)^2+ \td \theta^2+\sin^2\theta \td \phi^2]
\]
as $R\to \infty$.  From Example \ref{ex:TB} the period of $\tau/(2n)$ is $4\pi /p$, so  we see that this extension $(\hat M, \hat g)$ is ALE with the cross-section at infinity $S=L(p,1)$.
\end{example}

Motivated by these examples, we are now ready to define the class of conformal extensions of self-dual Einstein manifolds that we will study.

\begin{definition}\label{def:conformalALE}
A self-dual Einstein manifold $(M,g)$ admitting a Killing vector $\xi$, is said to be {\it conformally ALE with respect to $\xi$} if there is a complete K\"ahler manifold $(\hat M, \hat g, \hat \omega)$ such that $\xi$ extends to a Killing field of $(\hat M, \hat g)$, and:
\begin{enumerate}
        \item $M\subset \hat M$ and $M$ extends to a manifold with a boundary $\partial M$ within $\hat M$, that is, $M\cup \partial M\subset \hat M$
        \item the K\"ahler structure $(M, \hat{g}, \hat \omega)$ induced on $M$ coincides with that defined by $\xi$ 
        \item the Hamiltonian function $z$ defined by $\iota_\xi \hat \omega = - \td z$ is globally defined on $\hat M$ and
        \begin{enumerate} 
            \item $z>0$ on $M$ and coincides with that defined by $\xi$
            \item $z=0$ and $\td z \neq 0$ on $\partial M$
        \end{enumerate} 
        \item $(\hat M, \hat g)$ is ALE or order $\tau>1$ and $z<0$ on the ALE end.
\end{enumerate}
If $(M,g)$ is also toric, with $\xi$ in the span of the torus Killing fields,  we say $(M,g)$ is {\it toric conformally ALE} if additionally, $(\hat M, \hat g)$ is toric and $z$ is invariant under the torus symmetry.
\end{definition}

As shown above, the examples of hyperbolic space and self-dual negative Einstein Taub-bolt are toric conformally ALE.
In particular, note that our general definition of conformally ALE implies that $(M,g)$ is conformally compact with a defining function $z$ on $M$. Thus our definition  simply amounts to the addditional requirement that the conformally related K\"ahler metric extends to an ALE K\"ahler manifold.   We will find that this class is more amenable to classification than merely assuming conformally compact.

First, we will need the asymptotic form of the function $z$ in the ALE asymptotic region.  It is convenient to introduce explicit coordinates on the cross-section $S$. Choosing Hopf coordinates $(\theta,\phi_1,\phi_2)$, with $0\leq\theta\leq\pi/2$, we have $\hat{g}=g_0+O(r^{-\tau})$, where the asymptotic metric
\begin{align}\label{flatALE}
g_0 = \td r^2 + r^2 (\td \theta^2+\cos^2\theta\td\phi_1^2+\sin^2\theta\td\phi_2^2) 
\end{align}
is flat and the vector fields $\partial_{\phi_1},\partial_{\phi_2}$ generate the torus symmetry $T^2$.  If $S=S^3$ then $\phi_1, \phi_2$ are both $2\pi$ periodic, whereas if $S=L(p,q)$ then they are subject to identifications on a more general lattice.

\begin{proposition}\label{prop:zWALE}
Let $(\hat{M},\hat{g})$ be a toric ALE K\"ahler manifold, with K\"ahler form $\hat\omega$. Let $\partial_{\phi^{1}},\partial_{\phi^{2}}$ be a basis of toric Killing fields defined by Hopf coordinates at infinity, and consider the Killing vector $\xi=a_{1}\partial_{\phi^{1}}+a_{2}\partial_{\phi^{2}}$. Then we have the following asymptotic behaviour as $r\to\infty$:
\begin{align}
z ={}& \pm \tfrac{1}{2} r^2 (a_1\cos^2\theta+a_2\sin^2\theta) + O(r^{-\tau+1}).
\end{align}
\end{proposition}

\begin{proof}
We first derive the asymptotics of a K\"ahler form invariant under the torus symmetry for a toric ALE K\"ahler structure.  First we do this for the flat space model \eqref{flatALE}.  Introduce Cartesian coordinates $x^\mu$, $\mu=0,1,2,3$,
\[
x^0= r \cos \theta \cos\phi_1, \qquad x^1= r \cos\theta \sin \phi_1, \qquad x^2= r \sin\theta \cos\phi_2, \qquad x^3= r \sin\theta \sin\phi_2
\]
so $g_0= \td x^\mu \td x^\mu$. A basis of self-dual 2-forms is
\[
\omega^1= \td x^0 \wedge \td x^1+ \td x^2 \wedge \td x^3, \qquad  \omega^2= \td x^1 \wedge \td x^2+ \td x^0 \wedge \td x^3, \qquad \omega^3= \td x^1 \wedge \td x^3+ \td x^2 \wedge \td x^0
\]
where $\text{vol}_0 = \td x^0 \wedge \td x^1\wedge \td x^2 \wedge \td x^3$.  Hence we can expand the K\"ahler form $\hat{\omega}_0 = A_i \omega^i$ for some functions $A_i$ and since it must be parallel we deduce that the $A_i$ are all constants.  On the other hand, invariance under the torus symmetry implies that $A_2=A_3=0$  since only $\omega^1$ is invariant. Thus  $\hat{\omega}_0 = \pm \omega^1$, where reverting to polar coordinates,
\[
\omega^1= r \td r\wedge (\cos^2\theta \td \phi_1+ \sin^2\theta \td \phi_2) + r^2 \sin^2\theta\cos\theta \td \theta \wedge (\td \phi_2- \td \phi_1) \; .
\]
We deduce that
\[
 \iota_\xi \hat\omega_0 = -\td z_0, \qquad z_0 := \pm \tfrac{1}{2} r^2 ( a_1 \cos^2\theta+ a_2 \sin^2\theta)
\]
where $\xi$ is the Killing field defined above.

Now consider a toric ALE K\"ahler manifold. Clearly, the K\"ahler form must satisfy $\hat\omega = O(1)$, since $|\hat\omega|_{\hat{g}}= 2$.  Working in Cartesians, the ALE condition implies $\partial_\mu \hat g_{\nu \rho}= O(r^{-\tau-1})$ and hence the connection components are $O(r^{-\tau-1})$. Therefore,  $0=\hat\nabla_{\mu} \hat\omega_{\nu \rho}= \partial_\mu\hat\omega_{\nu \rho} + O(r^{-\tau-1})$ and hence $\partial_\mu  \hat\omega = O(r^{-\tau-1})$ which integrates to
$\hat{\omega} =\hat \omega_0 + O(r^{-\tau})$ where we have chosen the integration constants to coincide with the flat space K\"ahler form.
It follows that 
$z= z_0+  O(r^{1-\tau})$.
\end{proof}

\begin{cor}\label{cor:ALEend}
    If $(M, g)$ is a toric conformally ALE self-dual Einstein manifold the asymptotic ALE end of $\hat{M}$ corresponds to $z\to -\infty$.
\end{cor}

\begin{proof}
    First suppose the function $\beta:= a_1 \cos^2\theta+ a_2 \sin^2\theta$ is strictly positive or negative for all $\theta$. Then from Proposition \ref{prop:zWALE} it is clear that $r\to \infty$ implies $z\to \pm \infty$ depending on the sign of $\beta$. Hence from our definition of conformally ALE (property (4)) we must have $z\to -\infty$. Also, from Proposition \ref{prop:zWALE}  we have $z= O(r^2)$ as $r\to \infty$ and hence inverting this $r^{-2} = O(|z|^{-1})$ as $|z|\to \infty$, so we deduce that  $z \to -\infty$ implies $r\to \infty$.  
    
    If $\beta$ vanishes for some value of $\theta$, then Proposition \ref{prop:zWALE} implies that at this value $z= O(r^{1-\tau})$ so $z\to 0$ since in our definition we assume $\tau>1$ ;  but this means that $r\to \infty$ intersects the conformal boundary $z=0$ which is contrary to our definition of conformally ALE (property (4)). Therefore, we deduce that $\beta$ cannot vanish. 
\end{proof}

\begin{remark}\label{remark:beta}
    The above corollary shows that our definition restricts the choice of Killing field $\xi$ so that $\beta$ does not vanish. 
\end{remark}

We are now ready to determine the asymptotic form of the harmonic function $V$.  We will follow the method developed for Ricci-flat toric ALE Hermitian instantons~\cite{Araneda:2025uqo}.

\begin{proposition}\label{prop:asym}
Let $(M,g)$ be a toric conformally ALE self-dual Einstein manifold.
\begin{enumerate}
    \item
Weyl-Papapetrou coordinates in the asymptotic end are given by 
\begin{equation}
\rho = - z \, \sech(x) (1+ O(|z|^{-1}))  , \qquad \zeta = z  \tanh(x) (1+ O(|z|^{-1}) )   + O(\log |z|)\label{eqweylALE}
\end{equation}
as $z\to -\infty$. 

\item  Define coordinates $R:=\sqrt{\rho^2+ \zeta^2}$ and $X:= \sinh^{-1}(\zeta/\rho)$ so
\begin{equation}
    \rho= R\sech X , \qquad \zeta = R \tanh X \label{eq:RXcoords}
\end{equation}Then $R= -z +O(\log |z|)$ so the ALE end corresponds to $R\to \infty$ and the inverse is
 \begin{equation}
     z = -R+ O(\log R), \qquad x = -X + O(R^{-1} \log R)
 \end{equation}

\item 
The axisymmetric harmonic function $V$ in \eqref{Todmetric} can be written as 
\begin{align}\label{asymptoticsVALE}
V = - V_{0} +\tilde V, 
\end{align}
where $V_0$ is the basic solution \eqref{eq:V0} and $\tilde{V}$ is an axisymmetric harmonic function such that
    \begin{equation}
\partial_R \tilde{V} =  O(\log R/R)\; , \qquad  \partial_X \tilde{V}= O(\log R) \; ,
\end{equation}
    as $R\to \infty $.
\end{enumerate}
\end{proposition}

\begin{proof}
 First we prove (1). Rearrange \eqref{W} to
\[
u_z = \frac{2}{z} - \frac{2 W}{\ell^2 z}   \; .
\]
Next, note that $W^{-1}= | \xi |_{\hat{g}}^2$ and hence ALE implies
\[
W^{-1} = r^2 ( a_1^2 \cos^2\theta + a_2^2 \sin^2 \theta + O(r^{-\tau}))  \; .
\]
Now recall that our definition of conformally ALE with respect to $\xi$ implies that the function $\beta$ cannot vanish anywhere, see Remark \ref{remark:beta}. This implies that $a_1 a_2 \neq 0$ and hence $\alpha:= a_1^2 \cos^2\theta + a_2^2 \sin^2 \theta>0$.  Hence
\[
W = \frac{1}{\alpha r^2} (1+ O(r^{-\tau}) ) \; ,
\]
so in particular, from the proof of Corollary \ref{cor:ALEend}, we deduce that $W= O(|z|^{-1})$ as $z \to -\infty$.  Therefore, 
\begin{equation}
    u_z = \frac{2}{z} + O(z^{-2})  \;, \label{eq:uz}
\end{equation}
which can be integrated in the same way as for Ricci flat toric Hermitian ALE instantons (in that case the error is $O(z^{-3})$~\cite{Araneda:2025uqo}).  Integrating \eqref{eq:uz} and using the Toda equation \eqref{Todareduced} gives
\[
e^u = z^2 a^2 \text{sech}^2 (a x) (1+ O(|z|^{-1}))
\]
where $a>0$ is a constant and we have fixed an additive constant in $x$ (this corresponds to changing $V$ by an affine function of $\zeta$).  We can use the scaling freedom \eqref{eq:scaling} to fix $a=1$ which gives the claimed $\rho$ (taking the negative root since $z<0$ in this region).  Now integrating $\zeta_x= \tfrac{1}{2} (e^u)_z$ and $\zeta_z = - \tfrac{1}{2} u_x$ for $\zeta$ gives
\[
\zeta= z  \tanh(x) (1+ O(|z|^{-1}) )   + h(z)
\]
where $h(z)= O(\log |z|)$. 

We now prove (2). It immediately follows that $R^2= z^2+ O(z \log |z|)$ and hence we get the asymptotics for $R$.  Thus in particular, $R= O(|z|)$ and $R^{-1} = O(|z|^{-1})$, so  $z= O(R)$ and $z^{-1}= O(R^{-1})$.  This shows that the asymptotic end $z\to -\infty$ corresponds to $R\to \infty$ in Weyl-Papapetrou coordinates. This also gives the claimed inverse relation for $z$. To obtain the inverse for $x$, we can proceed by first noting that
\[
\sinh (-x) = \left( \frac{\zeta - h(z)}{\rho} \right) (1+ O(|z|^{-1}) ). 
\]
Then, using the general expansion,
\begin{equation}
    \sinh^{-1} ( a (1+b)) = \sinh^{-1}a + \frac{a}{\sqrt{1+a^2}} b + O(b^2) \; ,
\end{equation}
as $b \to 0$ where the error term is uniformly bounded in $a$, 
with $a=(\zeta - h(z))/\rho$ and $b= O(z^{-1})$ gives
\[
x = -\sinh^{-1}\left( \frac{\zeta - h(z)}{\rho} \right) + O(z^{-1})= -\sinh^{-1}\left( \frac{\zeta - h(z)}{\rho} \right) + O(R^{-1})  \; ,
\]
where we have used $a/\sqrt{1+a^2} = O(1)$ and $z^{-1}= O(R^{-1})$. To deal with the first term we can use a slight variant of the above expansion
\begin{equation}
    \sinh^{-1} ( a +b \sqrt{1+a^2} ) = \sinh^{-1}a +  b + O(b^2) \; ,
\end{equation}
where again the error terms are uniformly bounded in $a$, with $a= \zeta/\rho$ and $b=-h/R$, which gives
\[
\sinh^{-1}\left( \frac{\zeta - h(z)}{\rho} \right)= \sinh^{-1} \left( \frac{\zeta}{\rho} \right) -\frac{h}{R}+ O(h^2/R^2) = X + O( \log R /R)  \; ,
\]
where in the final equality we used $z= O(R)$ and $h= O(\log |z|)$. Putting it all together gives the claimed inverse relation for $x$.

Finally we prove (3). Changing from $(\rho, \zeta)$ to $(R, X)$ coordinates we have
\begin{equation}
\partial_R V = \frac{2z}{R} - \frac{2 \zeta x}{R} , \qquad \partial_X V= -\frac{2 \zeta z}{R} - \frac{2 \rho^2 x}{R} 
\end{equation}
and using the asymptotic formulae for $z, x$ in (2) we get
\[
\partial_R V= -2( 1- X \tanh X) + O(\log R/R), \quad \partial_X V = 2 R(\tanh X+X \text{sech}^2X) + O(\log R)
\]
Noting that  $V_0= 2R (1- X \tanh X)$ and defining $\tilde{V}= V+ V_0$ we deduce the result.
\end{proof}

\begin{remark}  
A toric conformally ALE self-dual Einstein manifold requires a choice of a Killing vector $\xi$ (in the span of the torus Killing fields) that defines the conformal K\"ahler structure and it is the associated K\"ahler metric $\hat g = z^2 g$ that is ALE. It is natural to ask what this means for the K\"ahler metrics associated to a different choice of Killing field $\tilde\xi= b_1\partial_\psi +b_2 \partial_y$ if $b_2 \neq 0$. Recall, if $\hat g$ is ALE then $\hat g\sim \td r^2+ r^2 \gamma$ and by Propositions \ref{prop:zWALE} and \ref{prop:asym} we have $-z \sim |\beta(\theta)| r^2$, $-z \sim R$. Then using the formulae for a change of Killing field in Section \ref{subsec:changeKVF} the K\"ahler metric defined by $\tilde\xi$, which is $\tilde{{g}}= \tilde{z}^2 g$,  must behave as
\[
\tilde{{g}} \sim   \frac{1}{b_2^2 \beta(\theta)^2}\left(  \td \tilde r^2+ \tilde r^2 \gamma \right)  \; ,
\]
 where $\tilde r\:= 1/r\to 0$. Thus we see that the ALE end corresponds to an orbifold point of the conformal K\"ahler structure with respect to all other choices of Killing field with $b_2\neq 0$.
\end{remark}

\section{Global analysis of toric self-dual Einstein instantons}
\label{sec:globalAnalysis}

\subsection{Rod structure and geometry near the axis}

In this section we will perform a global analysis of toric self-dual Einstein gravitational instantons $(M, g)$.  We first recall a technical definition.

\begin{definition}
    A gravitational instanton $(M,g)$ is toric if $M$ is simply connected and admits an (effective) isometric torus action with no discrete isotropy groups.
\end{definition}
 It has been shown that under these assumptions the orbit space $\mathcal{O}= M/T^2$ is a 2-dimensional simply connected manifold with boundaries and corners~\cite{Hollands:2007aj, OR}.  The Gram matrix of the toric Killing fields $\eta_i$, $i=1,2$, is $G_{ij}:= g(\eta_i, \eta_j)$, and the interior, the boundaries and corners of $\mathcal{O}$ correspond to points in $M$ where $G$ has rank-2, rank-1 and rank-0 respectively.

 The general local form of a toric self-dual negative Einstein metric is given by  Proposition \ref{prop:Tod}. As in Definition \ref{def:conformalALE}, we will assume that the conformal factor $z$ is strictly positive and globally defined on $M$. Therefore, the conformally related K\"ahler metric $\hat{g}= z^2 g$ is also globally defined on $M$. Hence, we can equivalently perform a global analysis of the toric scalar-flat K\"ahler manifold $(M, \hat g)$. In fact, we will perform the global analysis on the ALE extension $(\hat M, \hat g)$ introduced in Definition \ref{def:conformalALE}.  From Proposition \ref{prop:Tod} we deduce that the local form of the K\"ahler metric is
\begin{equation}\label{eq:SFKWP}
    \hat g = W^{-1}(\td\psi+F\td{y})^2+W\rho^2\td{y}^2+e^{2\nu}(\td\rho^2+\td\zeta^2),
\end{equation}
where $W, F, \nu$ are determined in terms of an axisymmetric harmonic function $V(\rho,\zeta)$ via equation \eqref{functionsTodmetric}. In the basis $\partial_\psi, \partial_y$ the Gram matrix is
\begin{equation}
\hat G = \left( \begin{array}{cc}  W^{-1} & W^{-1} F \\  W^{-1} F & W\rho^2 + W^{-1} F^2 \end{array} \right)
\end{equation}
and as observed in equation \eqref{eq:WPcoords}, the K\"ahler metric \eqref{eq:SFKWP} is automatically in Weyl-Papapetrou coordinates. For the class of asymptotics we consider this forms a global set of coordinates. 

\begin{prop}
Let $(\hat M, \hat g)$ be a toric ALE scalar-flat K\"ahler manifold. Then, Weyl-Papapetrou coordinates $(\rho, \zeta)$ are a global chart on the interior of $\mathcal{O}$.
\end{prop}
\begin{proof}
Toric ALE implies that the orbit space has an end diffeomorphic to a strip and Proposition \ref{prop:asym} implies the curve $\rho=\rho_0$ for large enough $\rho_0$ must lie in the asymptotic end. 
    The function $\rho$ is harmonic on the orbit space $(\mathcal{O}, \hat q)$, where $\hat q$ is given by \eqref{eq:WPcoords}. Therefore, the same proof as in the Ricci flat case applies~\cite{Araneda:2025uqo}.
\end{proof}   

The upshot of the previous proposition is that the interior of the orbit space is diffeomorphic to the half-plane $H= \{ (\rho, \zeta) \, | \, \rho>0 \}$. The boundary corresponds to $\rho=0$ and divides into intervals $I_i=(z_i, z_{i+1})$ for $i=0, \dots, n$ called `rods' that correspond to boundary segments separated by points $z_i$, $i=1, \dots n$ corresponding to the corners (we define $z_0=-\infty$ and $z_{n+1}=\infty$). The Gram matrix $\hat G$ is rank-1 on each rod $I_i$ with kernel spanned by $v_i$ called a rod vector, whereas the endpoints $\zeta=z_i$ are fixed points of the torus symmetry. The collection of data $\{ (I_i, v_i)_{i=0, \dots, n} \}$  is called the rod structure of the toric instanton. We will normalise the rod vectors $v_i$ to be $2\pi$ periodic, so in a $2\pi$-periodic basis of the $T^2$-action they can be represented by a pair of co-prime integers. The rod structure is then said to be admissible if in such a basis
\begin{equation}
    |\det (v_i, v_{i+1}) |=1    \label{eq:admissibility}
\end{equation}
for all $i=0, \dots, n-1$, which ensures the absence of orbifold singularities at the corners $z_i$. In general the metric $\hat g$ has a conical singularity over each rod $I_i$ as $\rho \to 0$ where $v_i$ vanishes, which is absent iff $v_i$ has $2\pi$-periodic orbits and 
\begin{equation}
 \frac{ | \td |v_i|^2|^2}{4 |v_i|^2}  \to 1, \qquad \zeta \in I_i, \quad \rho \to 0   \; .  \label{eq:conical}
\end{equation}
Then, the metric extends smoothly to $\rho=0$ if all metric components are smooth functions of $\rho^2$.

As we are assuming $z$ is a  smooth function on $\hat M$ we can expand it near any rod $I$ so
\begin{equation}
z = f(\zeta)+ O(\rho^2)  \label{eq:zexp}
\end{equation}
where $f(\zeta)>0$ on $M$.  Then, by the same argument as in~~\cite{Araneda:2025uqo}, we can integrate $z=\tfrac{1}{2} \rho V_\rho$ near each rod to determine the near axis behaviour of $V$ to be
\begin{equation}
V = f(\zeta) \log \rho^2 + g(\zeta)+ O(\rho^2) \;, \label{eqVnearaxis}
\end{equation}
where $f''(\zeta)=0$ by harmonicity of $V$.  Thus $f(\zeta)$ is a piecewise linear function with breaks at the corners $\zeta=z_i$. 

\begin{lemma}
Given a rod structure $\{ (I_i, v_i)_{i=0, \dots n} \}$ as above, 
\begin{equation}
f(\zeta) = A + \tfrac{1}{2} ( f_0'+ f_{n}') \zeta+ \sum_{i=1}^n a_i | \zeta - z_i| , \qquad a_i:= \tfrac{1}{2} (f_i'- f_{i-1}')    \label{eq:f}
\end{equation}
where $f_i' := f'(\zeta)|_{I_i}$ are the constant slopes on each rod and $A$ is a constant.
\end{lemma}
\begin{proof}
On each rod $I_k$, $k=0, \dots n$, the function must take the form 
\begin{equation}\label{eq:fIk}
    f(\zeta)|_{I_k}= f_k' \zeta+ c_k
\end{equation} for some constants $c_k$. Therefore, by continuity of $f$ at $\zeta=z_k$ for $k=1, \dots, n$, which is the lower endpoint of $I_k$ and the upper endpoint of $I_{k-1}$,  we must have $\lim_{\zeta\to z_k^+} f(\zeta)|_{I_k}= \lim_{\zeta\to z_k^-} f(\zeta)|_{I_{k-1}}$, which gives
\[
c_k- c_{k-1} = - 2 a_k z_k  \qquad k=1, \dots,  n  \; .
\]
The general solution to this recurrence equation is
\begin{equation}\label{eq:ck}
c_k = A+ \sum_{i=k+1}^n a_i z_i - \sum_{i=1}^k a_i z_i, \qquad k = 0, \dots, n
\end{equation}
where $A$ is a constant.

On the other hand, it is easily seen that \eqref{eq:f} evaluated on each rod gives
\[
f(\zeta)|_{I_k} = f_k' \zeta + A+ \sum_{i=k+1}^n a_i z_i - \sum_{i=1}^k a_i z_i, \qquad k = 0, \dots, n
\]
as required.
\end{proof}

We will now examine all metric components near the rods.
First assume that we are on a rod $I_i$ such that the slope $f'_i\neq 0$. Then we can expand the rest of the metric data as follows
\begin{align}
&W = \ell^2 \left( 1+ \frac{f(\zeta) V_{\zeta\zeta}}{2f'(\zeta)^2} + O(\rho^2) \right)    \label{eq:Waxis}
\\
&F= -\ell^2 \left( \frac{f(\zeta)}{f'(\zeta)} - \zeta + O(\rho^2) \right) \label{eq:Faxis}  \\
&e^{2\nu} =  W f'(\zeta)^2+ O(\rho^2)  \; .
\end{align}
Now on each rod we define $F_i:= F|_{\rho=0, \zeta\in I_i}$. Then, from the above expansion we deduce $\partial_\zeta F_i = 0$, so $F_i$ is a constant associated to each rod. Inspecting the metric \eqref{eq:SFKWP}, we deduce that if  we are on a rod $I_i$ such that $f_i'\neq 0$ and $W>0$ then the rod vector is
\begin{equation}
v_i = f_i' (\partial_y - F_i \partial_\psi) \; ,\label{eqrodvec}
\end{equation}
 where we have fixed the normalisation by the requirement \eqref{eq:conical} which ensures the absence of a conical singularity over $I_i$ provided the orbits of $v_i$ are $2\pi$ periodic. 

Now suppose we are on a rod $I_i$ where $f'_i=0$. Thus defining $f_i:= f(z_i)$ we deduce that $f_i= f(\zeta)= f_{i+1}$ for all $\zeta\in I_i$. The expansion of the metric data as $\rho \to 0$ is now
\[
W= \frac{2\ell^2 f_i}{\rho^2 g''(\zeta)} + O(1), \qquad F=O(1), \qquad  e^{2\nu} = \frac{\ell^2}{2} f_i g''(\zeta)+ O(\rho^2)
\]
provided $g''(\zeta)\neq 0$. It follows that the K\"ahler metric as $\rho \to 0$
\[
\hat{g}= \frac{\ell^2 f_i g''(\zeta)}{2} \left( \td \rho^2 + \frac{1}{\ell^4 f_i^2} \rho^2 \td \psi^2 \right) + O(\rho^2) \td \psi \td y+ O(1) \td y^2+ O(1) \td \zeta^2
\]
which has a conical singularity as $\rho\to 0$ unless the rod vector
\begin{equation}
v_i = \ell^2 f_i \partial_\psi
\label{eqrodvec2}
\end{equation}
is $2\pi$ periodic.

The above implies the following useful formulae which will be important in our regularity analysis.
\begin{lemma}
 If adjacent rods have $f_{i-1}', f_i' \neq 0$ then
\begin{equation}
F_i - F_{i-1} =- \ell^2 f_i \left( \frac{1}{f_i'}- \frac{1}{f_{i-1}'} \right)   \label{eq:Fdiff1}
\end{equation}

If $f_i'=0$, so $f_{i-1}'\neq 0, f_{i+1}'\neq 0$, then  
\begin{equation}
F_{i+1} - F_{i-1} =\ell^2 \left( z_{i+1}- z_i - f_i \left( \frac{1}{f_{i+1}'}- \frac{1}{f_{i-1}'} \right)\right)  \label{eq:Fdiff2}
\end{equation}
\end{lemma}

\begin{proof}
    First consider the case $f_{i-1}', f_i' \neq 0$. Then
    \[
    F_i - F_{i-1} = F(0, z_i^+) - F(0, z_i^-)
    \]
and the claim follows from using \eqref{eq:Faxis}.
    Now consider the case $f_i'=0$ so $f_{i-1}'\neq 0, f_{i+1}'\neq 0$. Then
    \[
    F_{i+1} - F_{i-1} = F(0, z_{i+1}^+) - F(0, z_i^-)
    \]
    and the claim follows from using \eqref{eq:Faxis}.
    \end{proof}

In order to impose smoothness at the fixed points $\zeta= z_i$ we  must require that  for any three consecutive rods $I_{j-1}, I_j, I_{j+1}$ both pairs of consecutive rod vectors form a basis for the torus action. The condition for this is that the basis $(v_{j-1}, v_j)$ is related to the basis $(v_j, v_{j+1})$ by a $GL(2,\mathbb{Z})$ transformation, which is equivalent to
\begin{equation}
v_{j-1}+ \epsilon_j v_{j+1}= l_j v_j  \label{eqreg}
\end{equation}
where $\epsilon_j=\pm 1$ and $l_j \in \mathbb{Z}$ for $j=1, \dots, n-1$.

There are now various cases to consider. 
First, if $f_{j-1}', f_j', f_{j+1}'\neq 0$, from the explicit form of the rod vectors, \eqref{eqreg} is equivalent to
\begin{align}
&f_{j-1}'+\epsilon_j f_{j+1}'= l_j f_j', \label{eqreg1} \\ &f_{j-1}' F_{j-1}+\epsilon_j f_{j+1}' F_{j+1}= l_j f_j' F_j  \; .
\end{align}
Now eliminate $l_j$ and use \eqref{eq:Fdiff1} to eliminate $F_j-F_{j-1}$ to get
\begin{equation}
f_{j+1}= \epsilon_j f_j \frac{f_j'- f_{j-1}'}{f_{j+1}'- f_j'}  \; .  \label{eqreg2}
\end{equation}
Thus the regularity conditions reduce to \eqref{eqreg1} and \eqref{eqreg2}.

Secondly, if $f_{j-1}' =0$ and $f_j', f_{j+1}'\neq 0$, then using the above rod vectors equation \eqref{eqreg} is equivalent to
\begin{equation}
\ell^2 f_{j-1}- \epsilon_j f_{j+1}' F_{j+1} =- l_j f_j' F_j, \qquad \epsilon_j f_{j+1}' = l_j f_j'
\end{equation}
and eliminating $l_j$ and using \eqref{eq:Fdiff2} for $F_{j+1}-F_j$ and the fact that $f_{j-1}=f_j$ (since $f'_{j-1}=0$) implies that the regularity conditions are given again by \eqref{eqreg2} and \eqref{eqreg1} with $f_{j-1}'=0$.   The same result is true if $f_{j-1}', f_j'\neq 0$ and $f_{j+1}'=0$.

The final case is if $f_j'=0$ in which case $f_{j-1}', f_{j+1}'\neq 0$.  Then \eqref{eqreg} is equivalent to 
\begin{align}
&f_{j-1}'+ \epsilon_j f_{j+1}'=0, \label{specialreg1} \\ &- f_{j-1}'F_{j-1} - \epsilon_j f_{j+1}' F_{j+1} = l_j f_j \ell^2   
\end{align}
and eliminating $f_{j-1}'$ and again elimnating $F_{j+1}-F_{j-1}$ using \eqref{eq:Fdiff2} implies
\begin{equation}
z_{j+1}- z_j = \frac{f_j}{f_{j+1}'} \left( 1+ \epsilon_j (1- l_j)\right) \; .  \label{specialreg2}
\end{equation}
Thus the regularity conditions in this case reduce to \eqref{specialreg1} and \eqref{specialreg2}.

\subsection{Uniqueness theorem}

In this section we will derive the general form of the harmonic function $V$ for the class of toric conformally ALE self-dual Einstein manifolds defined earlier. In particular, we will show that combining the near axis behaviour derived in the previous section, with the asymptotic behaviour we derived for conformally ALE manifolds, fixes the harmonic function $V$.

\begin{prop}
The harmonic function    
\begin{equation}\label{eq:Vmulti}
V = A \log \rho^2 + \sum_{i=1}^n a_i V_0(\rho, \zeta-z_i), \qquad \sum_{i=1}^n a_i=-1 \; ,
\end{equation}
is the most general axisymmetric harmonic function that satisfies the axis behaviour \eqref{eqVnearaxis}  with $f(\zeta)$ given by \eqref{eq:f} and the asymptotics in Proposition \ref{prop:asym} (up to an irrelevant additive constant). In particular, the slopes at infinity must be  $f_0'=-f_n'=1$.
\end{prop}

\begin{proof}
First, using \eqref{eq:V0axis}, we find that the most general harmonic function that satisfies the axis behaviour \eqref{eqVnearaxis} with $f(\zeta)$ given by \eqref{eq:f} is 
\[
V = A \log \rho^2+ \frac{1}{2} (f_n'+f_0') \zeta \log \rho^2  + \sum_{i=1}^n a_i V_0(\rho, \zeta-z_i)  + U  \; ,
\]
where $\sum_{i=1}^n a_i=\tfrac{1}{2}( f_n'-f_0')$ and $U$ is an arbitrary smooth axisymmetric harmonic function (to see this simply define $U$ by this equation and observe it is harmonic and smooth at the axis). 

Next, changing to $(R, X)$ coordinates \eqref{eq:RXcoords}, observe that
\begin{align}
 R\partial_R \left( V_0(\rho, \zeta)+ \sum_{i=1}^n a_i V_0(\rho, \zeta-z_i)\right) &= 2 R \left(1+\sum_{i=1}^n a_i \right) (1-X\tanh X) + O(1/R) \\ 
 \partial_X \left( V_0(\rho, \zeta)+ \sum_{i=1}^n a_i V_0(\rho, \zeta-z_i)\right) &= -2 R \left(1+\sum_{i=1}^n a_i \right) ( X \text{sech}^2 X + 
\tanh X) \nonumber \\ &+ 2\sum_{i=1}^n a_i z_i + O(1/R^2) \; ,
\end{align}
where the error terms are uniformly bounded in $X$.   On the other hand, 
\begin{align}
R\partial_R (\zeta \log \rho^2) &= 2R \tanh X (1+ \log R) - 2 R \tanh X \log \cosh X \\ &= 2 \tanh X (R \log R ) -2 R |X| \tanh X +  O(R)
\end{align}
where we used $\log \cosh X= |X|+ \log (1+e^{-2|X|})-\log 2$
and 
\begin{align}
\partial_X (\zeta \log \rho^2) & = 2 \text{sech}^2 X (R \log R)- 2R (\text{sech}^2 X \log \cosh X + \tanh^2 X) \\ & = 2 \text{sech}^2 X (R \log R) + O(R)  \; .
\end{align}
Therefore, putting everything together $\tilde{V}:= V+V_0$ satisfies
\begin{align}
R \partial_R \tilde{V} &= R \partial_RU+ (f_n'+f_0') \tanh X (R \log R )\nonumber  \\ &-2 R \tanh X \left\{   X\left(1+\sum_{i=1}^n a_i \right) + |X|\frac{1}{2} (f_n'+f_0') \right\}   + O(R) \\ 
\partial_X \tilde V &=  \partial_X U+ (f_n'+f_0') \text{sech}^2 X ( R \log R) +O(R)
\end{align}
Now recall the asymptotics in Proposition \ref{prop:asym}. First consider the condition $R \partial_R \tilde{V} = O(\log R)$;  fixing a large enough $R$ and using the fact that $R\partial_R U$ at fixed $R$ is a smooth and bounded function of $\tanh X$ (i.e. it is a smooth axisymmetric function on the constant $R$ spheres and hence must be bounded) implies that 
we must have 
\[
1+\sum_{i=1}^n a_i + \frac{1}{2} (f_n'+f_0')=0  , \qquad  1+\sum_{i=1}^n a_i -  \frac{1}{2} (f_n'+f_0')=0
\]
where the two equations arise  by taking $X\to \infty $ and $X\to -\infty$ and using the uniform bounds on the error terms. It follows that 
\[
f_0'=-f_n'=1 \; , \qquad \sum_{i=1}^n a_i=-1
\]
as claimed. Then the above asymptotic formulae reduce to
\[
R\partial_R \tilde V= R \partial_R U+ O(1/R), \qquad \partial_X \tilde V =\partial_XU + O(1)  \; .
\]
The first equation now implies  that $H=R\partial_R U=O(\log R)$ and since $H=R\partial_R U= x^i \partial_i U$ is itself a smooth axisymmetric harmonic function the bound on $H$ implies that $R\partial_R U$ is a constant. Thus we can write $R\partial_R U= a$ for constant $a$ and hence integrating and using harmonicity we get $U= a \log \rho + b X+ c$ where $b,c$ are constants. But $U$ is smooth on all of $\mathbb{R}^3$ so $a=b =0$ and hence $U$ is a constant.  The asymptotic conditions in Proposition \ref{prop:asym} are now satisfied.
\end{proof}

Therefore, we have shown that $V$ is given by a multipole solution \eqref{eq:Vmulti}.  We now need to perform a global analysis of this family of metrics.  It is useful to note that
\begin{align}
    z &= A+\sum_{i=1}^n a_i \sqrt{\rho^2+(\zeta-z_i)^2} \label{eq:zmulti} \; ,\\
V_{\zeta \zeta} &= -2 \sum_{i=1}^n \frac{a_i}{\sqrt{\rho^2+ (\zeta-z_i)^2}}  \; .  \label{eq:Vzz}
\end{align}
We may now deduce an important property of the multipole solution.

\begin{prop}\label{prop:multiVsol} Consider the multipole solution defined by \eqref{eq:Vmulti}.
For each fixed point $\zeta=z_i$ we must have $f_i a_i<0$. In particular, if all fixed points are inside the domain $z>0$ then:
\begin{enumerate}
    \item $a_i<0$ for all $i$ so  $f(\zeta)$  is concave
    \item $A>0$
    \item if $\rho>0$ then $W>0$ and $V_{\zeta\zeta}^2+V_{\zeta \rho}^2>0$ and are smoooth functions 
    \item if $\rho =0$,  $\zeta\in I_k$ and $f'_k\neq 0$ then $W>0$
\end{enumerate}
\end{prop}

\begin{proof}
First we prove (1). Expanding near the axis and using \eqref{eq:Waxis} and \eqref{eq:Vzz}, we find that away from any fixed point
\begin{equation}
W|_{\rho=0}= \ell^2 \left( 1- \frac{ f(\zeta)}{f'(\zeta)^2} \sum_{i=1}^n \frac{a_i}{|\zeta- z_i|}  \right) 
\end{equation}
and hence if we approach a fixed point $\zeta \to z_j$ we must have $f_ja_j<0$ to ensure $W$ does not become negative.  Therefore, if all fixed points have $z>0$, since $z=f_i$ at a fixed point $\zeta=z_i$,  we must have that $f_i>0$ for all $i=1, \dots, n$, then $a_i<0$ for all $i$ and hence $f(\zeta)$ is concave.   

Now we deduce (2). Since $a_i<0$ we see from \eqref{eq:zmulti} that $z<A$ everywhere on the half plane $\rho \geq 0$. Hence since by assumption the region $z>0$ is non-empty we must have $A>0$.

Next consider (3). Since $a_i<0$ for all $i$ then \eqref{eq:Vzz} implies $V_{\zeta\zeta}>0$ and hence $V_{\zeta\zeta}^2+V_{\zeta \rho}^2>0$. Next, from the explicit form of $W$ in \eqref{functionsTodmetric} and $V$ in \eqref{eq:Vmulti} we can write
\begin{equation}
   W= \frac{4 \ell^2 B}{ \rho^2 (V_{\zeta\zeta}^2+ V_{\zeta\rho}^2)} \; ,
\end{equation}
where upon simplification we find
\begin{equation}\label{eq:B}
    B= - A \sum_{i=1}^n \frac{a_i}{d_i} - \sum_{i<j} \frac{ a_i a_j z_{ij}^2}{d_id_j} \; ,
\end{equation}
and for convenience we have defined $d_i:=\sqrt{\rho^2+(\zeta-z_i)^2}$ and $z_{ij}:=z_i -z_j$.  Now, since  $f_i= A+ \sum_{j=1}^n a_j |z_{ij}|>0$ and $a_i<0$ for each $i$, we  deduce that
\begin{equation}
    B> \sum_{i<j} \frac{a_ia_j |z_{ij}| ( d_i+d_j- |z_{ij}|)}{d_i d_j}  >0  \; ,
\end{equation}
where in the final inequality we have used the triangle inequality $d_i+d_j>|z_{ij}|$ if $\rho>0$.  Hence $W>0$ for all $\rho>0$ as required.

Finally consider (4). By restricting to $\rho=0$ we find that on the interior of $I_k$ we have
\begin{equation}
    W_{\rho=0}= \frac{\ell^2 B|_{\rho=0}}{f_k'^2}
\end{equation}
where the restriction $B|_{\rho=0}$ takes the same form as \eqref{eq:B} with $d_i=|\zeta-z_i|$. Then using $f_i>0$ and arguing as above we get
\begin{equation}
    B_{\rho=0}> \sum_{i<j} \frac{a_ia_j |z_{ij}| r_{ij}}{|\zeta-z_i| |\zeta-z_j|} , \qquad r_{ij}:= |\zeta-z_i|+|\zeta-z_j|- |z_{ij}| \; .
\end{equation}
If $i<j\leq k$ or $k<i<j$ then $\zeta\in I_k$ implies $r_{ij}>0$, whereas if $i\leq k<j$ then $\zeta\in I_k$ implies $r_{ij}=0$. Therefore, running the sum over all $i<j$ we deduce that $B_{\rho=0}>0$ for $\zeta\in I_k$ as required.
\end{proof}

The previous proposition shows that the metric in Proposition \ref{prop:Tod}, for $V$ given by the multipole solution \eqref{eq:Vmulti} with $f(\zeta)$ concave, is smooth and invertible everywhere on the half-plane $\rho>0$ and hence away from the axis. Furthermore, part (4) shows that the rod vectors on any rod with $f_i'\neq 0$ are indeed given by \eqref{eqrodvec}. Therefore, it remains to analyse regularity at the axis. Before turning to this, we will check that the solution has the required asymptotics.

It is straightforward to verify that this family of solutions is actually ALE. Using \eqref{functionsTodmetric} with $V$ given by the multipole solution \eqref{eq:Vmulti} and changing to $(R, X)$ coordinates \eqref{eq:RXcoords}, we find
\begin{align}
    W &= \frac{\ell^2 A}{R} + O(R^{-2}), \qquad \tfrac{1}{4}\rho^2 (V_{\zeta\zeta}^2+ V_{\zeta \rho}^2) = 1+ O(R^{-2}) \\  F &= \ell^2 (A \tanh X- \sum_{i=1}^n a_i z_i) + O(R^{-1})  \;.
\end{align}
Now, changing coordinates to
\begin{equation}
        r^2= 4 \ell^2 A R, \qquad \cos\theta = \tanh X, \qquad \bar{\psi}= \frac{\psi- \ell^2 (\sum_{i=1}^n a_i z_i) y } {\ell^2 A}
\end{equation}
we deduce that \eqref{eq:SFKWP}  is
\begin{equation}
\hat{g} = \td r^2+ \frac{r^2}{4} ( \td \theta^2+ \sin^2 \theta \td y^2 + ( \td \bar\psi+ \cos\theta \td y)^2 )  + O(r^{-2}) \; ,
\end{equation}
where the error terms are with respect to the norm of the asymptotic metric. Thus $(\hat M, \hat{g})$ is indeed ALE (with $\tau=2$) with a lens space cross-section $S$ for suitable identifications of the torus angles $(\bar\psi, y)$. The corresponding self-dual Einstein $g= z^{-2} \hat{g}$ for $z>0$,  is defined on the region
\begin{equation}
\sum_{i=1}^n |a_i| \sqrt{\rho^2+(\zeta-z_i)^2}<A  \; ,
\end{equation}
where we used \eqref{eq:zmulti},
which corresponds to the interior of a `generalised hyper-ellipse' in the half-plane $\rho \geq 0$. Thus $(M,g)$ is conformally compact and the conformal boundary $z=0$ corresponds to the boundary of the hyper-ellipse, \begin{equation}
    \sum_{i=1}^n |a_i| \sqrt{\rho^2+(\zeta-z_i)^2}=A  \;, 
\end{equation} in the half-plane with $\partial M=S$ given by the same lens space as in the ALE end.

\subsection{Classification of smooth instantons}
We will now consider the axis regularity conditions for the class of multipole solutions \eqref{eq:Vmulti} for which all fixed points are in the domain $z>0$, so in particular $f_i>0$ for all $i=1, \dots, n,$.  As shown in Proposition \ref{prop:multiVsol}, this implies that $A>0$ and that $f(\zeta)$ is concave so $f_{i}'>f'_{i+1}$ for all $i$ with $f_0'=1$ and $f_n'=-1$, which gives
\begin{equation}\label{eq:fALE}
    f(\zeta)= A+\sum_{i=1}^n a_i |\zeta-z_i| 
\end{equation}
with $a_i=\tfrac{1}{2}(f_i'-f_{i-1}')<0$.

First note for $n=1$ the solution is simply
\begin{equation}
    V= A \log \rho^2- V_0(\rho, \zeta-z_1)
\end{equation}
which from Example \ref{ex:hyperbolic} and \ref{ex:multi} we see corresponds to hyperbolic space (the parameter $z_1$ is trivial and can be removed by translating $\zeta$ and the parameter $A$ can be set to any positive constant using the scaling freedom in Remark \ref{rem:scaling2}).

Thus we now assume $n\geq 2$ henceforth. Given three consecutive rods $I_{j-1}, I_j, I_{j+1}$, the regularity conditions \eqref{eqreg2} and \eqref{specialreg1} together with $a_i<0$ imply $\epsilon_i=1$. Therefore, \eqref{eqreg} simplifies and the rod vector satisfy
\begin{equation}\label{eq:vvlsimp}
    v_{j-1}+ v_{j+1}= l_j v_j
\end{equation}
and $l_j \in \mathbb{Z}$ for $j=1, \dots, n-1$. The regularity conditions then reduce to 
\begin{equation}
    f'_{j-1}+ f_{j+1}'= l_j f_j'  \label{eq:simpreg1}
\end{equation}
and
\begin{align}
    \frac{f_{j+1}}{f_j} &=\frac{f_j'- f_{j-1}'}{f_{j+1}'- f_j'}  \qquad \text{if}  \quad f_j'\neq 0  \label{eq:simpreg2}\\
    z_{j+1}- z_j &= \frac{f_j}{f_{j+1}'}( 2- l_j) \qquad \text{if} \quad f_j'=0    \label{eq:simpreg3}
\end{align}
In fact, by noting that $f_{j+1}= f_j+ f_j'(z_{j+1}-z_j)$  the  conditions \eqref{eq:simpreg2} and \eqref{eq:simpreg3} and can written in the unified form 
\begin{equation}
z_{j+1}- z_j = \frac{f_j (l_j-2)}{f_j'- f_{j+1}'}     \label{eq:simpreg4}
\end{equation}
for any $f_j'$. Therefore, the regularity conditions reduce  \eqref{eq:simpreg1} and \eqref{eq:simpreg4}.   It is also worth noting that the rod vectors \eqref{eqrodvec} and \eqref{eqrodvec2} can be written in the unified form
\begin{equation}
    v_i= f_i' \partial_y+ \ell^2 c_i \partial_\psi   \label{eqrodunified}
\end{equation}
for any $f_i'$, where we have used \eqref{eq:fIk} to obtain $F_i=-\ell^2 c_i/f_i'$ for the case $f_i'\neq 0$ and $f_i=c_i$ for the case $f_i'=0$.

\begin{lemma}
    The integers $l_j>2$ for all $j=1, \dots, n-1$.
    \label{thm:l>2}
\end{lemma}

\begin{proof} This immediately follows from \eqref{eq:simpreg4} together with concavity $f_j'-f'_{j+1}>0$ and $f_j>0$.
\end{proof}


Now consider the $n=2$ case so we have $1=f_0'>f_1'>f_2'=-1$. Then \eqref{eq:simpreg1} together with the fact that $l_1>2$ implies that $f_1'=0$ and hence
\begin{equation}
 f(\zeta)= A- \tfrac{1}{2}|\zeta-z_1| -\tfrac{1}{2}|\zeta-z_2|   \; .
\end{equation}
Then \eqref{eq:simpreg3} reduces to $z_2-z_1= f_1 (l_1-2)$ and since $f_1=A-\tfrac{1}{2} (z_2-z_1)$ we can solve for
\begin{equation}\label{parametern=2}
    z_2-z_1= \frac{2A (l_1-2)}{l_1} \; .
\end{equation}
 The scaling freedom in Remark \ref{rem:scaling2} allows us to set the constant $A $ to any positive value.
Thus we obtain a family of solutions specified by an integer $l_1>2$. Choosing a basis of rod vectors $v_1=(1,0), v_2=(0,1)$ we have  $v_0=l_1 v_1-v_2= (l_1, -1)$ so the topology at infinity is $L(l_1, 1)$.  It can be checked that this solution is precisely the self-dual negative Einstein Taub-bolt solution in Example \ref{ex:TB}. 
To see this, note that \eqref{eq:zTB} gives $z_2-z_1=2\sqrt{c}$ and $A=2n/\ell^2$, and \eqref{nTB} gives $c=(p-2)^2/[4\ell^2(p-1)]$ (where recall $p\geq3$), so we indeed get \eqref{parametern=2} with $p\equiv l_1$.

\begin{lemma}\label{lemma:Mjinvert}
    Given $f_0^\prime$ and $f_n^\prime$ and $l_i>2$ for $i=1, \dots, n-1$, the remaining slopes $f_1^\prime, \cdots, f_{n-1}^\prime$ are uniquely determined in terms of the integers $l_1, \cdots, l_{n-1}$.
\end{lemma}
\begin{proof}
    Given $f_0^\prime$ and $f_n^\prime$, equation \eqref{eq:simpreg1} can be re-written in terms of a tridiagonal matrix as 
    \begin{align}
        \begin{bmatrix}
            l_1 & -1 & 0 & 0 & \cdots & 0 \\
            -1 & l_2 & -1 & 0 & \cdots & 0 \\
            0 & -1 & l_3 & -1 & \cdots & 0 \\
            \vdots & \vdots & \vdots & \vdots & \ddots & \vdots \\
            0 & 0 & 0 & 0 &\cdots & l_{n-1}
        \end{bmatrix}\begin{bmatrix}
            f_1^\prime \\
            f_2^\prime \\
            f_3^\prime \\
            \vdots \\
            f_{n-1}^\prime
        \end{bmatrix} &= \begin{bmatrix}
            f_0^\prime \\
            0 \\
            0 \\
            \vdots \\
            f_n^\prime
        \end{bmatrix}.
        \label{eq:matrixRegularity}
    \end{align}
    If the matrix on the left is invertible, then the lemma holds immediately. To show this, define a sequence of matrices, for $j=1, \dots, n-1$,
    \begin{align}
        M_j &= \begin{bmatrix}
            l_j & -1 & 0 & 0 & \cdots & 0 \\
            -1 & l_{j+1} & -1 & 0 & \cdots & 0 \\
            0 & -1 & l_{j+2} & -1 & \cdots & 0 \\
            \vdots & \vdots & \vdots & \vdots & \ddots & \vdots \\
            0 & 0 & 0 & 0 &\cdots & l_{n-1}
        \end{bmatrix}.
    \end{align}
    By expanding along the top row observe that
    \begin{align}
        \mathrm{det}(M_j) &= l_j\mathrm{det}(M_{j+1}) - \mathrm{det}(M_{j+2}).
    \end{align}
    This recursion is accompanied by the initial condition $\mathrm{det}(M_{n-1}) = l_{n-1}$ and we define $\mathrm{det}(M_n) = 1$ to get the correct value $\mathrm{det}(M_{n-2})=l_{n-1} l_{n-2}-1$. The recursion then implies
    \begin{align}
        \mathrm{det}(M_j) - \mathrm{det}(M_{j+1}) &= (l_j - 2)\mathrm{det}(M_{j+1}) + \mathrm{det}(M_{j+1}) - \mathrm{det}(M_{j+2}).
    \end{align}
    Thus if $\mathrm{det}(M_{j+1}) > 0$ and $\mathrm{det}(M_{j+1}) - \mathrm{det}(M_{j+2}) > 0$, then $l_j>2$ implies $\mathrm{det}(M_j) - \mathrm{det}(M_{j+1}) > 0$ and $\mathrm{det}(M_j) > 0$ as well. Since $\mathrm{det}(M_{n-1}) > \mathrm{det}(M_n) > 0$, by induction $\mathrm{det}(M_j)$ is increasing as $j$ decreases and positive for all $j=1, \dots, n-1$. Hence, in particular $\mathrm{det}(M_1) > 0$ and so equation \eqref{eq:matrixRegularity} has unique solution.
\end{proof}

\begin{lemma}\label{lemma:endslopes}
    Suppose $n\geq 3$. Then $f_1'>0$ and $f'_{n-1}<0$.
\end{lemma}

\begin{proof}
    Suppose $f_1'<0$.  This implies $f_2<f_1$ and hence \eqref{eq:simpreg2} gives 
\[
 \frac{1-f_1'}{f_1'-f_2'}<1   \qquad \implies \qquad 2f_1'>1+f_2'>0  \; ,
\]
where the final equality follows from  $f_2'>f_n'=-1$, which contradicts $f'_1<0$.  If $f_1'=0$ then \eqref{eq:simpreg1}  implies  $1+ f_2'=0$ which also contradicts $f_2'>-1$.  Hence we must have $f_1'>0$.

Now suppose $f_{n-1}'>0$. Then $f_n>f_{n-1}$ and so  \eqref{eq:simpreg2} implies
\[
\frac{f_{n-2}'- f_{n-1}'}{1+ f_{n-1}'} >1   \qquad \implies \qquad 2f_{n-1}'< f_{n-2}'-1<0 \; ,
\]
where the final equality is from  concavity, so we obtain a contradiction. If $f_{n-1}'=0$ then \eqref{eq:simpreg3}  implies  $f_{n-2}'-1=0$  which contradicts $1>f_{n-2}'$. Hence we must have $f_{n-1}'< 0$.
\end{proof}

It is convenient to perform the regularity analysis with respect to a K\"ahler structure defined by a Killing field $\tilde\xi= b_1 \partial_\psi+ b_2 \partial_y$ with $b_2\neq 0$ that necessarily has closed orbits. For simplicity we fix $\ell=1$ henceforth. In particular, we transform to a basis of Killing fields $(\tilde\xi, \tilde \eta)= (\partial_\psi, \partial_y) L^{-1}$ where $L\in GL(2, \mathbb{R})$ as in Section \ref{subsec:changeKVF} and denote all quantities defined with respect to $\tilde{\xi}$ with a tilde.  From Example \ref{ex:multipolechange} and Example \ref{ex:multi} we immediately deduce that the solution with respect to the new basis remains a multipole solution and is given by 
\begin{equation}
    \tilde{V} = aA|b_2| V_0(\tilde{\rho}, \tilde{\zeta} - c) + a\sum_{i = 1}^na_i|b_1 + b_2z_i|V_0(\tilde{\rho}, \tilde{\zeta} - \tilde{z}_i) \; ,
\end{equation}
where $a=| \det L |$.

Now using the near axis behaviour \eqref{eq:V0axis},  we deduce that $\tilde{V}= \tilde{f}(\tilde{\zeta})\log \tilde{\rho}^2+ O(1)$ as $\tilde{\rho}\to 0$ where  
\begin{equation}\label{eq:ftilde}
    \tilde{f}(\tilde\zeta)= a A |b_2| |\tilde\zeta- c|+\sum_{i=1}^n a_i a | b_1+ b_2 z_i| |\tilde{\zeta}- \tilde z_i|  \; ,
\end{equation}
which is still piecewise linear but with an extra break at $\tilde{\zeta}=c$.  It turns out that a particularly convenient choice is to consider the K\"ahler structure defined by the rod vector on $I_n$, 
\begin{equation}
    \tilde{\xi}= v_n \qquad \implies \qquad  b_1=c_n, \quad b_2=-1  \; ,
\end{equation}
where we have used that $v_n= c_n \partial_\psi- \partial_y $ which follows from \eqref{eqrodunified} and $f_n'=-1$.  Then, the $w=\zeta+i\rho$ plane maps to the $\tilde{w}=\tilde{\zeta}+ i \tilde{\rho}$ plane via 
\begin{equation}
    \tilde{w} = c+\frac{1}{a(c_n-w)}  \; .
\end{equation}
In particular, restricted to the axis this gives
\begin{equation}
\tilde{\zeta}|_{\rho=0} = c+\frac{1}{a(c_n-\zeta)}  \; ,
\end{equation}
so the interval $(-\infty, c_n)$ on the $\zeta$-axis is diffeomorphic to the interval $(c, \infty)$ on the $\tilde{\zeta}$-axis with  $\tilde{\zeta}|_{\rho=0}$  a monotonically strictly increasing function of $\zeta$.  But $f_n=c_n-z_n$ so $f_n>0$ implies $c_n-z_i>0$ for all $i=1, \dots, n$, which means all fixed points $\zeta=z_i$ lie in the interval $(-\infty, c_n)$ and get mapped to fixed points $\tilde{z}_i=\tilde{\zeta}|_{(0, z_i)}$ in the interval $(c, \infty)$ such that the ordering is preserved $\tilde{z}_i<\tilde{z}_{i+1}$.  For uniformity of notation it is convenient to define $\tilde{z}_0:= c$ and $\tilde{z}_{n+1}= \infty$ and rods $\tilde{I}_i= (\tilde{z}_i, \tilde{z}_{i+1})$ for $i=0, \dots n$ and add an additional rod $\tilde I_{-1}= (-\infty, c)$. 

We  now find \eqref{eq:ftilde} simplifies to,
\begin{equation}
   \tilde{f}(\tilde\zeta) = \sum_{i=0}^n \tilde{a}_i |\tilde \zeta- \tilde{z}_i| \; ,
   \end{equation}
   where
   \begin{equation} \label{eq:newai}
   \tilde{a}_0 := a A, \qquad \tilde{a}_i :=a_i a (c_n-z_i)\qquad  \text{if $1\leq i \leq n$} \; .  
\end{equation}
Using $\sum_{i=1}^n a_i=-1$, we deduce that
\begin{align}
    \sum_{i=0}^n \tilde a_i &= a \left( A- c_n - \sum_{i=1}^n a_i z_i \right) = 0   \; , \\  \sum_{i=0}^n \tilde{a}_i \tilde{z}_i &=-1+ ac \left( A- c_n - \sum_{i=1}^n a_i z_i \right) =-1 \; ,
\end{align}
where the final equalities follow from the fact that $c_n= A-\sum_{i=1}^n a_i z_i$ in equation \eqref{eq:ck}. This immediately implies that $\tilde{f}|_{\tilde{\zeta}<c} =-1$ and $\tilde{f}|_{\zeta>\tilde{z}_n}=1$, so in particular the slopes of $\tilde f$ on $\tilde{I}_{-1}$ and $\tilde{I}_n$ vanish. Furthermore, on each rod $\tilde{I}_i$ we can write $\tilde{f}|_{\tilde{I}_i} =  \tilde{f}_i' \tilde\zeta+ \tilde{c}_i$ where 
\begin{align}
\tilde{f}_{-1}' &= \tilde{f}_{n}'=0, \qquad \tilde{c}_{-1}=-\tilde{c}_n=-1   \label{eq:tildeends}\\ 
    \tilde{f}'_i &= a ( f_i' c_n+ c_i), \qquad \tilde{c}_i= - f'_i- c \tilde{f}'_i \; , \qquad i=0, \dots , n
    \label{eq:fAndCTilde}
\end{align}
where $c_i$ are given by \eqref{eq:ck}. We can write $\tilde{a}_i= \tfrac{1}{2} (\tilde{f}'_i- \tilde{f}'_{i-1})$ for $i=0, 1, \dots, n$, so we deduce that the concavity of $f$ condition $a_i<0$ for $i=1, \dots, n$ together with \eqref{eq:newai} implies that $\tilde{a}_i<0$ for $i=1, \dots, n$ and hence
\begin{equation}
   \tilde{f}'_i- \tilde{f}'_{i-1}<0  \; ,\qquad i=1, \dots, n \; ,
\end{equation}
so $\tilde{f}$ is concave for $\tilde{\zeta}>\tilde{z}_0$.

Next, we fix the remaining freedom in $L\in GL(2, \mathbb{R})$ to set other second Killing field $\tilde{\eta}= v_{n-1}$, which corresponds to 
\begin{equation}
    L^{-1}= \begin{pmatrix} c_n & c_{n-1} \\ -1 & f'_{n-1} \end{pmatrix}  \; ,
    \label{eq:aInverse}
\end{equation}
where we have used \eqref{eqrodunified} for $i=n-1$.  Thus, recalling $a=| \det L|$,
\begin{equation}
    a= \frac{1}{c_n f_{n-1}'+ c_{n-1}} \; ,
\end{equation}
where we have used that $c_n f_{n-1}'+c_{n-1}= f_n(1+f_{n-1}')>0$, which follows from $c_i= -f_i' F_i$ and $f_{n-1}'>-1=f_n'$ and \eqref{eq:Fdiff1}. Using \eqref{eq:fAndCTilde} this value of $a$ in particular implies that
\begin{equation}
    \tilde{f}_{n-1}'=1   \; .
\end{equation}
Now using the linear transformation $(\partial_\psi, \partial_y)=(v_{n}, v_{n-1}) L$  we find that the rod vectors \eqref{eqrodunified} relative to the new basis are
\begin{equation}
   v_i= \tilde{f}_i' v_{n-1} + \tilde{c}_{i} v_n \; 
\end{equation}
for $i=0, \dots, n$, where $\tilde{c}_{n-1}=0$ and we have compared to  \eqref{eq:fAndCTilde} which fixes
\begin{equation}
   c= - f'_{n-1}  \; .
    \label{eq:c}
\end{equation}
Since the rod vectors $v_i$ have closed orbits with $2\pi$ period by construction, it must be that
\begin{equation}
    \tilde{f}'_i=m_i, \qquad \tilde{c}_i =n_i
\end{equation} 
for coprime integer pairs $(m_i,n_i)$ where $i=0, \dots, n$. Thus on each rod we have
\begin{equation}
    \tilde{f}|_{\tilde{I}_i} = m_i \tilde{\zeta} + n_i
\end{equation}
where for uniformity in notation we have defined $m_{-1}=0$ and $n_{-1}=-1$ by comparing to \eqref{eq:tildeends}.

Relative to the basis $(v_{n-1}, v_{n})$ we can thus represent the rod vectors for $i=0, \dots n$ as
\begin{equation}
    v_i=(m_i, n_i) ,
\end{equation}
where our choices give $m_n = 0$, $m_{n-1} = 1$, $n_n = 1$ and $n_{n-1} = 0$.  We can now impose regularity at the axis in this new basis. The remaining $(m_i, n_i)$ are fixed by the regularity condition \eqref{eq:vvlsimp}, which in the new variables in equivalent to the recursion relations
\begin{align}
    m_{i-1} = l_im_i - m_{i+1}\,\,\,\mathrm{and}\,\,\,n_{i-1} = l_in_i - n_{i+1} \; ,
    \label{eq:recursion}
\end{align}
for $i=1, \dots, n-1$. Recall we showed that $l_i > 2$ in Lemma \ref{thm:l>2}. We now show that the regularity conditions do not  impose any further restrictions, i.e., they hold for arbitrary choices of $l_i > 2$.

Firstly, the recursion implies
\begin{align}
    m_{i-1} - m_i = (l_i - 2)m_i + m_i - m_{i + 1}.
\end{align}
Thus, if $m_i \geq 0$ and $m_i - m_{i + 1} > 0$, then $m_{i - 1} - m_i > 0$ too. Since $m_n = 0$ and $m_{n-1} = 1$, it therefore follows by induction that 
\begin{equation}
    m_{i - 1} - m_i > 0, 
\end{equation} for all $i=1, \dots, n-1$. In particular, the concavity of $\tilde{f}$ holds as required.  Next, observe that equation \eqref{eq:recursion} implies
\begin{align}
    n_{i+1}m_i - n_im_{i + 1} = n_im_{i-1} - n_{i - 1}m_i
\end{align}
for $i=1, \dots, n-1$.
Since $n_nm_{n-1} - n_{n-1}m_n = 1$, we get 
\begin{align}
    n_{i+1}m_i - n_im_{i + 1} = 1
    \label{eq:unitDet}
\end{align}
for all $i=0, \dots, n-1$, so the admissibility condition is satisfied \eqref{eq:admissibility}. Furthermore, by Bezout's lemma, $(m_i, n_i)$ are indeed pairwise coprime.

Next, continuity of $\tilde{f}$ at $\tilde\zeta=\tilde{z}_i$ for $i=0, \dots, n$, requires $m_i\tilde{z}_i + n_i = m_{i-1}\tilde{z}_i + n_{i - 1}$, which fixes
\begin{align}
\label{eq:zitilde}
    \tilde{z}_i &= \frac{n_i - n_{i - 1}}{m_{i-1} - m_i}
\end{align}
for all $i=0, \dots, n$.
Hence, we find the rod lengths satisfy
\begin{align}
\label{eq:rodlengths}
    \tilde{z}_{i + 1} - \tilde{z}_i &= \frac{l_i - 2}{(m_i - m_{i + 1})(m_{i - 1} - m_i)} > 0 \qquad i=1, \dots, n-1 \; , \\
    \tilde{z}_1-\tilde{z}_0 &= \frac{m_0-m_1+1}{m_0 (m_0-m_1)} >0 \; ,
\end{align}
automatically. Finally, from equation \eqref{eq:unitDet} it also follows that for $i=1, \dots, n$,
\begin{align}
    \tilde{f}_i = m_i\tilde{z}_i + n_i = \frac{1}{m_{i - 1} - m_i} > 0  \; ,
\end{align}
as required for all corners to be inside the original manifold, instead of the conformally ALE extension.  Therefore, the integers $(m_i, n_i)$ are not subject to any further constraints and we have proven the following.

\begin{proposition}\label{prop:rodvectors}
The rod vectors $v_i=(m_i, n_i)$ for $i=0, \dots, n$ are given by
\begin{equation}
v_0=(m_0, n_0),  \quad \dots \quad , \quad v_{n-1}=(1,0),\quad  v_n=(0,1)
\end{equation}    
where the integers $(m_i, n_i)$ are determined by the recursions \eqref{eq:recursion} and $(l_1, \dots, l_{n-1})$ are any set of integers satisfying $l_i>2$.   In particular, the topology at infinity is $L(m_0, n_0)$.
\end{proposition}

We can now deduce that the full solution relative to the torus Killing fields $\tilde{\xi}, \tilde{\eta}$ can be written as the multipole solution
\begin{equation}
    \tilde{V}= \sum_{i=0}^n \tilde{a}_i V_0(\tilde{\rho}, \tilde{\zeta}- \tilde{z}_i), \qquad   \tilde a_i= \tfrac{1}{2}(m_i -m_{i-1}), \quad i=0, \dots, n,  \label{eq:finalV}
\end{equation}
where the fixed points $\tilde{z}_i$ are determined by \eqref{eq:zitilde} and the integers $(m_i,n_i)$ are determined by integers $l_i>2, i=1, \dots, n-1$ via Proposition \ref{prop:rodvectors} with $m_{-1}=0, n_{-1}=-1$.  Therefore, the full solution is parameterised by just the integers $l_i>2 , i=1, \dots, n-1$ and furthermore there is a solution for any such choice of $l_i$. Now, by Example \ref{ex:multi} these solutions correspond to the  multipole solutions
\begin{equation}
    \mathcal{F}= \sum_{i=0}^n \tilde{a}_i \mathcal{F}_0(\tilde{\rho}, \tilde{\zeta}-\tilde{z}_i)
\end{equation}
where $\tilde{a}_i, \tilde{z}_i$ are as in \eqref{eq:finalV}. In fact, these are precisely the Calderbank-Singer multipole solutions which were constructed directly in the $(v_{n-1}, v_n)$ basis~\cite{Calderbank:2002gy}. This completes the proof of Theorem \ref{thm:main}.

\end{document}